\documentclass[11pt,a4paper]{amsart}
\usepackage{amsmath, amssymb}
\usepackage{amsfonts}
\usepackage{mathrsfs}
\usepackage[arrow,matrix,curve,cmtip,ps]{xy}
\usepackage{amsthm}
\usepackage{amscd}
\usepackage[pdftex,colorlinks,linkcolor=magenta,citecolor=cyan]{hyperref}
\usepackage[dvipsnames]{xcolor}
\usepackage{pst-node}
\usepackage{tikz-cd} 
\usepackage{cancel}
\usepackage[colorinlistoftodos]{todonotes}

\allowdisplaybreaks

\theoremstyle{plain}
\newtheorem{theorem}{{Theorem}}[subsection]
\newtheorem{lemma}[theorem]{Lemma}
\newtheorem{proposition}[theorem]{Proposition}

\newtheorem{notation}[theorem]{Notation}
\newtheorem{remark}[theorem]{Remark}
\newtheorem{definition}[theorem]{Definition}

\newtheorem{question}{Question}
\newtheorem{conjecture}{Conjecture}

\numberwithin{equation}{section}
\numberwithin{theorem}{subsection}
\newcount\colveccount
\newcommand*\colvec[1]{
	\global\colveccount#1
	\begin{pmatrix}
		\colvecnext
	}
	\def\colvecnext#1{
		#1
		\global\advance\colveccount-1
		\ifnum\colveccount>0
		\\
		\expandafter\colvecnext
		\else
	\end{pmatrix}
	\fi
}

\newcommand{\sL}{\mathsf{L}}

\newcommand{\hit}{\mathsf{Hit}}
\newcommand{\ano}{\mathsf{Ano}}

\newcommand{\Spec}{\mathsf{Spec}}

\newcommand{\sA}{\mathsf{A}}

\newcommand{\sP}{\mathsf{P}}
\newcommand{\sK}{\mathsf{K}}

\newcommand{\sQ}{\mathsf{Q}}
\newcommand{\sT}{\mathsf{T}}

\newcommand{\sG}{\mathsf{G}}
\newcommand{\sV}{\mathsf{V}}
\newcommand{\sW}{\mathsf{W}}
\newcommand{\sF}{\mathsf{F}}
\newcommand{\sM}{\mathsf{M}}
\newcommand{\sN}{\mathsf{N}}

\newcommand{\asQ}{\mathsf{AQ}}
\newcommand{\asP}{\mathsf{AP}}
\newcommand{\SL}{\mathsf{SL}}
\newcommand{\ASL}{\mathsf{ASL}}
\newcommand{\AG}{\mathsf{AG}_\tR}
\newcommand{\cA}{\mathcal{A}}

\newcommand{\cX}{\mathcal{X}}
\newcommand{\cY}{\mathcal{Y}}

\newcommand{\cV}{\mathcal{V}}
\newcommand{\acV}{\mathcal{AV}}

\newcommand{\tAd}{\mathtt{Ad}}
\newcommand{\tad}{\mathtt{ad}}

\newcommand{\tC}{\mathtt{C}}
\newcommand{\tB}{\mathtt{B}}

\newcommand{\tR}{\mathtt{R}}

\newcommand{\texp}{\mathtt{exp}}
\newcommand{\tJd}{\mathtt{Jd}}
\newcommand{\tM}{\mathtt{M}}

\newcommand{\tth}{\mathtt{h}}
\newcommand{\te}{\mathtt{e}}
\newcommand{\tu}{\mathtt{u}}
\newcommand{\N}{\mathbb{N}}
\newcommand{\R}{\mathbb{R}}

\newcommand{\sHom}{\mathsf{Hom}}

\newcommand{\defeq}{\mathrel{\mathop:}=}

\newcommand{\flow}{\mathsf{U}\Gamma}
\newcommand{\cflow}{\widetilde{\mathsf{U}_0\Gamma}}
\newcommand{\bdry}{\partial_\infty\Gamma}
\newcommand{\fg}{\mathfrak{g}}
\newcommand{\fa}{\mathfrak{a}}

\newcommand{\fk}{\mathfrak{k}}

\newcommand{\fp}{\mathfrak{p}}
\newcommand{\fn}{\mathfrak{n}}
\newcommand{\fm}{\mathfrak{m}}

\newcommand{\fsl}{\mathfrak{sl}}

\title[Proper actions]{Deformation of Fuchsian representations and proper affine actions}
\author{Sourav Ghosh}
\address{Ashoka University, Rajiv Gandhi Education City, Rai, Sonipat, Haryana 131029, India}
\email{sourav.ghosh@ashoka.edu.in, sourav.ghosh.bagui@gmail.com}
\date{\today}

\thanks{The author acknowledge support from Ashoka University annual research grant.}
\begin{document}
	
	\begin{abstract}
		The main goal of this article is to generalize Mess' work and using results from Labourie--Wentworth, Potrie--Sambarino and Smilga, to show that inside Hitchin representations, infinitesimal deformations of Fuchsian representations of a cocompact surface group do not act properly along the directions corresponding to the sum of a mixed odd differential and a $2m$-differential for any $1\leq m \leq \lfloor\frac{n}{2}\rfloor$. 
		
		In the process, we introduce affine versions of cross ratios and triple ratios. We introduce Margulis invariants and relate them with affine crossratios and infinitesimal Jordan projections. We obtain a general equivalent criterion for existence of proper affine actions in terms of the structure of the Margulis invariant spectra. Also, using a stability argument we show the existence of proper affine actions of non-abelian free groups whose linear part is a Hitchin representation. 
	\end{abstract}
	
	\maketitle
	\tableofcontents
	
	\newpage
	
	\section{Introduction}
	
	\subsection{Background and main results}
	
	The goal of this article is to study deformations of Fuchsian representations in the space of Hitchin representations and relate it to questions regarding proper affine actions of surface groups. Following the celebrated Bieberbach theorems \cite{B1,B2}, Auslander attempted to classify all cocompact proper affine actions \cite{Aus}. Auslander's failed attempt was later rechristened by Fried--Goldman \cite{FG} as the Auslander conjecture which states that
	\begin{conjecture}
		The fundamental group of a compact complete affine manifold is virtually polycyclic.
	\end{conjecture}
	This conjecture still remains unproved in the general case but in lower dimensions (see Fried--Goldman \cite{FG}, Abels--Margulis--Soifer \cite{AMS2}, Tomanov \cite{Tom}) results confirming the conjecture has been attained. Recently, an excellent survey of this topic has appeared in Danciger--Drumm--Goldman--Smilga \cite{DDGS}. Later, Milnor \cite{Milnor} asked if one could disprove the weakened Aulander conjecture by dropping the compactness assumption and Margulis \cite{Margulis1,Margulis2} in a celebrated work constructed examples of non-abelian free groups $\Gamma$ acting properly on $\R^3$ via affine actions. In order to show properness of the action Margulis introduced certain invariants $\tM(\gamma)$, called the \textit{Margulis invariants}, for all $\gamma\in\Gamma$. The Margulis invariant of $\gamma$ is defined as the standardised projection of the translational part of $\gamma$ on the unit eigenspace of $\gamma$ (see Definition \ref{def.marginv} for more details). 
	One can reformulate Margulis' results using isomorphisms of low dimensional Lie groups, in the following two ways:
	\begin{enumerate}
		\item There exists non-abelian free groups $\Gamma\subset\mathsf{SO}_{2,1}\ltimes\R^3$ which act properly on $\R^3$.
		\item There exists non-abelian free groups $\Gamma\subset\SL_2(\R)\ltimes\fsl_2(\R)$ which act properly on $\fsl_2(\R)$. 
	\end{enumerate}
	Subsequently, Abels--Margulis--Soifer \cite{AMS} constructed free subgroups of $\mathsf{SO}_{2n,2n-1}\ltimes\R^{4n-1}$ which act properly on $\R^{4n-1}$. Moreover, they showed that there does not exist any free subgroup of $\mathsf{SO}_{2n+1,2n}\ltimes\R^{4n+1}$ which act properly on $\R^{4n+1}$. These results can be seen as generalizations of the first reformulation of Margulis' work. Recently, generalizations of the second reformulation of Margulis' work was obtained by Smilga \cite{Smilga, Smilga3, Smilga4} and Danciger--Gu\' eritaud--Kassel \cite{DGK3}. When $\sG$ is a noncompact semisimple real Lie group, Smilga \cite{Smilga} constructed free subgroups of $\sG\ltimes\fg$ which act properly on $\fg$. These results were further generalized by Smilga \cite{Smilga3,Smilga4} to construct free subgroups of $\sG\ltimes_\tR\sV$ which act properly on $\sV$, under some conditions on the representation $\tR$ via which the noncompact semsimple Lie group $\sG$ acts on $\sV$. In particular the respresentation must admit zero as a weight. The linear parts of all these examples contain no parabolic elements, although Drumm \cite{Drumm2} showed that one could construct examples of free subgroups of $\mathsf{SO}_{2,1}\ltimes\R^3$ which act properly on $\R^3$ and whose linear part contains parabolic elements. All these constructions relied on a suitable generalization of the Margulis invariants. In general, roughly speaking the Margulis invariant of an element $(g,Y)\in\sG\ltimes_\tR\sV$ is the standardized projection of $Y$ on the unit eigenspace of $g$ (see Definition \ref{def.marginv} for more details). In this article, we introduce affine versions of cross ratios and relate them with Margulis invariants. Suppose $(g,Y)\in\sG\ltimes_\tR\sV$ be such that its action on the space of affine parabolic subspaces (see Definitions \ref{def.linpara} and \ref{def.affpara}) has an attracting fixed point $\sA_+$, a repelling fixed point $\sA_-$ and $\sA_\pm$ are transverse to each other. Moreover, suppose $\sA$ is another affine parabolic space which is transverse to both $A_\pm$ and $(g,Y)\sA$ is transverse to $\sA$. We denote the \textit{affine cross ratio} of the four mutually transverse affine parabolic subspaces by $\beta(\sA_+,\sA_-,(g,Y) \sA, \sA)$ (see Definition \ref{def.affcr} for more details). We generalize results from Ghosh \cite{Ghosh3} and prove that:
	
	\begin{theorem}[see Theorem \ref{thm.alphabeta}]
		Suppose $(g,Y)$ is in $\sG\ltimes_\tR\sV$ and $\sA_\pm$, $\sA$ are as above. Then the affine cross ratio $\beta$ and the Margulis invariant $\tM$ satisfy the following relation:
		\[\beta(\sA_+,\sA_-,(g,Y) \sA, \sA)=\tM(g,Y)+\tM((g,Y)^{-1}).\]
	\end{theorem}
	
	In fact, we also introduce affine triple ratios and relate them with affine cross ratios (see Definition \ref{def.afftr} and Proposition \ref{prop.cr}). Similar results relating Margulis invariants and cross ratios has also been independently obtained by Andr\' es Sambarino \cite{Samba2}.
	
	Recently, there has been a lot of activity trying to characterize existence of proper affine actions in terms of the Margulis invariant spectra.	These works show that examples whose linear part contain no parabolic elements can be characterized more easily than the examples constructed by Drumm. In particular, Goldman--Labourie--Margulis \cite{GLM} gave an equivalent criterion (see also Danciger--Gu\'eritaud--Kassel \cite{DGK2}), in terms of Margulis invariants, for existence of proper affine actions of free groups inside $\mathsf{SO}_{2,1}\ltimes\R^3$ whose linear part does not contain parabolic elements. Building upon this and the work of Goldman--Labourie \cite{GL}, Ghosh \cite{Ghosh2} related the equivalent criterion of proper affine actions with the notion of an Anosov representation. Anosov representations were introduced by Labourie \cite{Labourie} to give a characterization of representations in the Hitchin components \cite{Hit}. A fundamental property of Anosov representations is that they are stable under small deformations. Let $\Gamma$ be a word hyperbolic group, $\tau:\Gamma\to\SL_2(\R)$ be any injective homomorphism and $\iota:\SL_2(\R)\to\sG$ be the irreducible representation. Then any representation of the form $\iota\circ\tau$ is called a Fuchsian representation and the connected components of the representation variety which contain Fuchsian representations are called Hitchin components. We denote these components by $\hit(\Gamma,\sG)$. Any representation contained in the Hitchin components is called a Hitchin representation. Subsequently, Ghosh--Treib \cite{GT} generalized the result of Goldman--Labourie--Margulis and gave an equivalent criterion, in terms of Margulis invariants and Anosov representations, for existence of proper affine actions of word hyperbolic groups inside $\mathsf{SO}_{2n,2n-1}\ltimes\R^{4n-1}$ whose linear part does not contain parabolic elements. We extend these results in this paper by replacing $\mathsf{SO}_{2n,2n-1}\ltimes\R^{4n-1}$ with $\sG\ltimes_\tR\sV$, where $\sG$ is a real split semisimple algebraic Lie group with trivial center and $\tR:\sG\to\SL(\sV)$ is a faithful irreducible algebraic \textit{non-swinging} representation admitting zero as a weight (see Remark \ref{rem.weight}).
	\begin{theorem}[see Theorem \ref{thm.proper}]
		Let $\Gamma$ be a word hyperbolic group and $(\rho,u):\Gamma\to\sG\ltimes_\tR\sV$ be an injective homomorphism such that $\rho\in\ano(\Gamma,\sG,\tR)$ (see Definition \ref{def.ano}). Then the action of $(\rho,u)(\Gamma)$ on $\sV$ is not proper if and only if there exists a diverging sequence $\{\gamma_n\}_{n\in\N}$ inside $\Gamma$ such that the Margulis invariants $\tM(\rho(\gamma_n),u(\gamma_n))$ stay bounded. 
	\end{theorem}  
	We note that Kassel--Smilga \cite{KaSm} have a similar theorem in a related, but slightly different, setting.
	
	Moreover, when $\sV=\fg$ and $\tR$ is the adjoint representation we generalize results of Goldman--Margulis \cite{GM} and show in Proposition \ref{prop.derivative} that the Margulis invariants are infinitesimal Jordan projections. Similar results relating Margulis invariants and Jordan projections has also been independently obtained by Andr\' es Sambarino \cite{Samba2} and Kassel--Smilga \cite{KaSm}. Furthermore, we use the stability of Anosov representations under small deformations to prove the following existence result (similar result but in a different setting was previously obtained by Ghosh \cite{Ghosh3}): 
	
	\begin{proposition}[see Proposition \ref{prop.main}]
		Let $\Gamma$ be a free group and $(\rho,u):\Gamma\to\SL_n(\R)\ltimes\fsl_n(\R)$ is an injective homomorphism such that $\rho$ is a Fuchsian representation. Then there exists a neighborhood $U$ of  $[\rho]$ in $\hit(\Gamma,\SL_n(\R))$ such that for any $[\varrho]\in U$ there exists some non empty open set $W_{[\varrho]}\subset\sT_{[\varrho]}\hit(\Gamma,\SL_n(\R))$ and for any $w\in W_{[\varrho]}$, the group $(\varrho,w)(\Gamma)$ acts properly on $\fsl_n(\R)$.
	\end{proposition}
	
	On contrary, Mess \cite{Mess} showed that the fundamental group of a compact orientable surface without boundary and of genus atleast two admits no proper affine action on $\R^3$. Later, Goldman--Margulis \cite{GM} gave an alternate proof of Mess' result. This result can be reformulated in the following two ways:
	\begin{enumerate}
		\item Suppose $\Gamma$ is the fundamental group of a compact orientable surface without boundary and of genus atleast two and $(\rho,u):\Gamma\to\mathsf{SO}_{2,1}\ltimes\R^3$ is an injective homomorphism such that $\rho$ is Fuchsian. Then $(\rho,u)(\Gamma)$ does not act properly on $\R^3$.
		\item Suppose $\Gamma$ is the fundamental group of a compact orientable surface without boundary and of genus atleast two and $(\rho,u):\Gamma\to\SL_2(\R)\ltimes\fsl_2(\R)$ be an injective homomorphism. Then $(\rho,u)(\Gamma)$ does not act properly on $\fsl_2(\R)$.
	\end{enumerate}
	Subsequently, Labourie \cite{Labourie2} showed that if $\Gamma$ is the fundamental group of a compact orientable surface without boundary and of genus atleast two and $(\rho,u):\Gamma\to\mathsf{SO}_{n+1,n}\ltimes\R^{2n+1}$ is such that $\rho$ is Fuchsian, then $(\rho,u)(\Gamma)$ does not act properly on $\R^{2n+1}$. Recently, Danciger--Zhang \cite{DZ} and Labourie \cite{Labourie3} generalized Labourie's result and showed that if $\Gamma$ is the fundamental group of a compact orientable surface without boundary and of genus atleast two and $(\rho,u):\Gamma\to\mathsf{SO}_{n+1,n}\ltimes\R^{2n+1}$ is such that $\rho$ is Hitchin, then $(\rho,u)(\Gamma)$ does not act properly on $\R^{2n+1}$. Both these results generalize the first reformulation of Mess' result. In this article, we generalize the second reformulation of Mess' result and show the following:
	\begin{theorem}[see Theorem \ref{thm.main}]
		Let $\Gamma$ be a cocompact surface group and $(\rho,u):\Gamma\to\SL_n(\R)\ltimes\fsl_n(\R)$ is an injective homomorphism such that $\rho$ is a Fuchsian representation. Also, suppose $[u]$ in $\sT_{[\rho]}\hit(\Gamma,\SL_n(\R))$ correspond to the sum of a mixed odd differential and a $2m$-differential for any $1\leq m\leq \lfloor\frac{n}{2}\rfloor$. Then $(\rho,u)(\Gamma)$ does not act properly on $\fsl_n(\R)$.
	\end{theorem}
	As a main ingredient in proving the above result we use results of Smilga \cite{Smilga4}, Labourie--Wentworth \cite{LW} and Potrie--Sambarino \cite{PotSam}. More generally, we are interested in knowing the answer to the following question:
	\begin{question}
		What are all $[u]$ in $\sT_{[\rho]}\hit(\Gamma,\SL_n(\R))$ for which $(\rho,u)(\Gamma)$ does not act properly on $\fsl_n(\R)$?
	\end{question}
	The naive guess would be that $(\rho,u)(\Gamma)$ does not act properly on $\fsl_n(\R)$ for all $[u]$ in $\sT_{[\rho]}\hit(\Gamma,\SL_n(\R))$ but we are unable to prove this result. An immediate reason is due to the existing restrictions in the statement of Corollary 1.4 of Potrie--Sambarino \cite{PotSam} but a more major reason is the fact that no relation is known to exist (at least by us) between Margulis invariant of a single element along different $k$-differentials.

	\subsection{Plan of the paper}
	The paper consists of four main sections.
	
	In the first section we introduce preliminary notions which are needed later to state and prove our results. The first section is divided into two subsections. In subsection \ref{sec.RLG} we recall some known results about representations of Lie groups. In subsection \ref{sec.AR} we define Anosov representations in general and state some of their properties relevant to our article. 
	
	In the second section we study affine deformations of Anosov representations. This section is again divided into two parts. In subsection \ref{sec.NS} we introduce neutral sections, a concept first introduced by Goldman--Labourie--Margulis \cite{GLM} to set the ground of defining affine equivalent of an Anosov representation. And, in subsection \ref{sec.PAA} we define appropriate affine versions of Anosov representations.
	
	In the third section we introduce certain affine invariants which play a central role in our article. The third section is divided into three subsections. In subsection \ref{sec.ACTR} we introduce affine cross ratios and affine tripe ratios and show their inter-relationship. In subsection \ref{sec.MI} we introduce Margulis invariants and relate them with affine cross ratios and infinitesimal Jordan projections. In subsection \ref{sec.CS} we prove a limit formula and obtain convexity results about the Margulis invariant spectra and the Jordan-Margulis invariant spectra. 
	
	In the fourth section we prove results related to proper affine actions. This section is divided into three parts. In subsection \ref{sec.CPA} we prove one of our main results and give a general criterion about proper affine actions and derive some other useful criteria about proper affine actions from it. In subsection \ref{sec.DFG} we use the stability of Anosov representations and show existence of proper affine actions of free surface groups with Hitchin linear part. Finally, in subsection \ref{sec.DSG} we prove the most important result of this article. We use results from \cite{LW}, \cite{PotSam} and Smilga \cite{Smilga4} to obtain non-existence results of proper affine actions of cocompact surface groups with Fuchsian linear part.

	\subsection{Acknowledgements}
	I would like to thank Fran\c cois Labourie for his insight, encouragement and in particular for funding me to visit him at Nice. I would also like to thank him for his suggestions to improve the exposition of this article and Pritam Ghosh and C. S. Rajan for many fruitful discussions. Moreover, I would like to heartily thank all those people who were beside me during the last couple of years. In particular, my parents Swapan Ghosh, Purnima Ghosh, my brother Gourab Ghosh, my well wishers Sandeepan Parekh, Nishant Chandgotia, Arghya Mondal, Neeraja Sahasrabudhe, Pallavi Panda, Indira Chatterji, Fran\c cois Labourie, Mahan Mj, Indranil Biswas, Jean-Marc Schlenker, Anna Wienhard, Cagri Sert and Rajendra Bhatia. Without their timely support this article could never have been possible.
	
	Lastly, we acknowledge that Andr\' es Sambarino, Fanny Kassel and Ilia Smilga have independently derived results similar to some of the results obtained in this article.
	
	\section{Preliminaries}\label{sec.1}

	\subsection{Representation of Lie groups}\label{sec.RLG}

	Let $\Gamma$ denote a word hyperbolic group with finitely many generators, $\sG$ be a real split semisimple algebraic Lie group with trivial center and without compact factors. We denote the Lie algebra of $\sG$ by $\fg$ and the identity element of $\sG$ by $e$. Now for any $g,h\in \sG$ we consider $\tC_g(h):=ghg^{-1}$ and denote the differential of $\tC_g(h)$ at $e$ by $\tAd_g$. The map $\tAd: \sG \to \SL(\fg)$ is a homomorphism. The differential of $\tAd$ at $e$ is denoted by $\tad$ and the Killing form on $\fg$ is denoted by $\tB$. We fix a Cartan involution $\theta:\fg\to\fg$ and consider the corresponding decomposition $\fg=\fk\oplus\fp$ where $\fk$ (respectively $\fp$) is the eigenspace of eigenvalue 1 (respectively -1). Let $\fa$ be a maximal abelian subspace of $\fp$. We consider the space of linear forms on $\fa$ and denote it by $\fa^*$. Also, for all $\alpha\in\fa^*$ we define
	\[\fg^\alpha:=\{X\in\fg\mid \tad_H(X)=\alpha(H)X\text{ for all } H\in\fa\}.\]
	Any element $\alpha\in\fa^*$ is called a \textit{restricted root} if and only if both $\alpha\neq0$ and $\fg^\alpha\neq0$. We denote the set of all restricted roots by $\Sigma$. As $\fg$ is finite dimensional, it follows that $\Sigma$ is finite. Moreover, we have
	\[\fg=\fg^0\oplus\bigoplus_{\alpha\in\Sigma}\fg^\alpha.\]
	We choose a connected component of $\fa\setminus\cup_{\alpha\in\Sigma}\ker(\alpha)$ and denote it by $\fa^{++}$. The closure of $\fa^{++}$ is denoted by $\fa^+$. Let $\sK\subset \sG$ (respectively $\sA\subset \sG$) be the connected subgroup whose Lie algebra is $\fk$ (respectively $\fa$) and let $\sA^+:=\exp{(\fa^+)}$. The Lie group $\sK$ is a maximal compact subgroup of $\sG$.
	
	The subset of restricted roots which take positive values on $\fa^+$ is denoted by $\Sigma^+$. We note that $\Sigma=\Sigma^+\sqcup-\Sigma^+$ and define
	\[\fn^\pm:=\bigoplus_{\pm\alpha\in\Sigma^+}\fg^\alpha\]
	and note that they are nilpotent subalgebras. Let $\sN\subset\sG$ be the connected subgroup whose Lie algebra is $\fn^+$. Finally, we denote the centralizer of $\fa$ inside $\sK$ by $\sM$ and denote the Lie subalgebra of $\fg$ coming from $\sM$ by $\fm$. We note that $\fg=\fn^+\oplus\fg^0\oplus\fn^-$ and $\fg^0=\fa\oplus\fm$.
	
	\begin{remark}
		If $\sG$ is split then $\fm$ is trivial and $\sM$ is a finite group (see Theorem 7.53 of \cite{Knapp}).
	\end{remark}

	Let $\sV$ be a finite dimensional real vector space with $\dim_\R\sV>1$ and $\tR:\sG\to\SL(\sV)$ be a faithful irreducible algebraic representation. Hence, we obtain $d_e\tR:\fg\to\mathfrak{sl}(\sV)$. Suppose $\lambda\in\fa^*$. We define
	\[\sV^\lambda:=\{X\in\sV\mid d_e\tR(H)(X)=\lambda(H)X\text{ for all } H\in\fa\}.\] 
	If $\sV^\lambda\neq0$, then we call $\lambda\in\fa^*$ a \textit{restricted weight} of the representation $\tR$. We denote the set of all restricted weights by $\Omega$. As $\sV$ is finite dimensional, it follows that $\Omega$ is finite. Moreover, we note that
	\[\sV=\bigoplus_{\lambda\in\Omega}\sV^\lambda.\]
	Let $\Omega$ be the set of all restricted weights. Then the following quotient group is called the \textit{restricted Weyl group}:
	\[W := \{g\in\sG\mid g\sA g^{-1}=\sA\}/ \{g\in\sG\mid ghg^{-1}=h \text{ for all } h\in\sA\}.\]
	We denote the longest element of the restricted Weyl group by $\omega_0$ i.e. $\omega_0(\Sigma^+)=-\Sigma^+$. Suppose $X\in\fa$. We consider the following sets:
	\[  \Omega^\pm(X):=\{\lambda\in\Omega\mid\pm\lambda(X)>0\}\text{ , }
	\Omega^{\pm,0}(X):=\{\lambda\in\Omega\mid\pm\lambda(X)\geqslant0\},\]
	and define $\Omega^0(X):=\Omega^{+,0}(X)\cap\Omega^{-,0}(X)$. Any two vectors $X,Y\in\fa$ are said to be of the \textit{same type}, denoted by $X\sim Y$, if and only if $\Omega^\pm(X)=\Omega^\pm(Y)$.

	\begin{definition}
		Suppose $X\in\fa$. Then 
		\begin{enumerate}
			\item $X$ is called \textit{generic} if and only if $\Omega^0(X)\subset\{0\}$,
			\item $X$ is called \textit{symmetric} if and only if $-\omega_0(X)=X$,
			\item $X$ is called \textit{extreme} if and only if
			\[\{\omega\in W\mid \omega X \sim X\}=\{\omega\in W\mid \omega X = X\}.\]
		\end{enumerate}
	\end{definition}
	
	\begin{remark}\label{rem.weight}
		In order to guarantee existence of proper affine actions, henceforth, we only consider representations, $\tR$, for which $\sV^0$ is nontrivial. 
		
		Moreover, we assume that $\tR$ admits an extreme, symmetric and generic vector and we fix one such vector $X_0$ in $\fa^+$. Such a representation $\tR$ is called a non-swinging representation.
	\end{remark}
	
	Let $\sV^{\pm,0}$ be the vector subspace generated by all $\sV^\lambda$ such that $\lambda\in\Omega^{\pm,0}(X_0)$, $\sV^\pm$ be the vector subspace generated by all $\sV^\lambda$ such that $\lambda\in\Omega^\pm(X_0)$ and $\sV^0:=\sV^{+,0}\cap\sV^{-,0}$. We observe that
	\[\sV=\sV^+\oplus\sV^0\oplus\sV^-.\]
	Let $\Sigma^+(X_0):=\{\alpha\mid\alpha(X_0)\geqslant0\}$. We denote the vector subspace generated by $\fg^0$ and all $\fg^\alpha$ such that $\alpha\in\pm\Sigma^+(X_0)$ by $\fp^{\pm,0}$.
	
	Let $g\in \sG$. We say $g$ is \textit{elliptic} if $g$ is conjugate to some element in $\sK$. We say $g$ is \textit{hyperbolic} if $g$ is conjugate to some element in $\sA$. We say $g$ is \textit{unipotent} if $g$ is conjugate to some element in $\sN$.
	
	\begin{theorem}[Jordan decomposition]
		Let $\sG$ be a real semisimple algebraic Lie group of noncompact type with trivial center. Then for any $g\in \sG$, there exist unique $g_\te,g_\tth,g_\tu\in \sG$ such that the following hold:
		\begin{enumerate}
			\item $g=g_\te g_\tth g_\tu$,
			\item $g_\te$ is elliptic, $g_\tth$ is hyperbolic and $g_\tu$ is unipotent,
			\item the elements $g_\te,g_\tth,g_\tu$ commute with each other.
		\end{enumerate}
		
	\end{theorem}
	\begin{definition}\label{def.jordan}
		Let $\sG$ be a real semisimple algebraic Lie group of noncompact type with trivial center and let $g\in \sG$. Then the \textit{Jordan projection} of $g$, denoted by $\tJd_g$, is the unique element in $\fa^+$ such that $g_\tth$ is a conjugate of $\exp{(\tJd_g)}$.
	\end{definition}
	
	\begin{definition}
		Let $\sG$ be a real semisimple algebraic Lie group of noncompact type with trivial center and let $g\in \sG$. Then $g$ is called \textit{loxodromic} if and only if $\tJd_g\in\fa^{++}$.
	\end{definition}
	
	We use Lemmas 6.32 and 6.33 (ii) of \cite{BQ} and Appendix V.4 of \cite{Whit} (see also \cite{Tit}) to observe that $\tJd$ is continuous. Moreover, for all $g,h\in\sG$ we have $\tJd_{hgh^{-1}}=\tJd_g$. The analytic version of the implicit function theorem (see Theorem 6.1.2 of Krantz--Parks \cite{KraPa}) imply that $\tJd$ vary analytically over loxodromic elements. Moreover, when $\sG=\SL(d,\R)$ then $\tJd(g)=(\log|\lambda_1(g)|,\dots,\log|\lambda_d(g)|)$, where each $\lambda_i(g)$ is an eigenvalue of $g$ such that $|\lambda_1(g)|\geq \dots \geq|\lambda_d(g)|$.

	\begin{theorem}
		Let $\sG$ be a real semisimple Lie group of noncompact type. Then for any $g\in \sG$ there exists $k_1,k_2\in \sK$, not necessarily unique, and a unique $\kappa(g)\in \fa^+$ such that $g=k_1\exp{(\kappa(g))}k_2$ i.e.
		$\sG=\sK\sA^+\sK$.
	\end{theorem}
	\begin{definition}
		The map $\kappa:\sG\to\fa^+$ is called the Cartan projection.
	\end{definition}
	
	We use Corollary 5.2 of Benoist \cite{Ben2} to note that
	\[\lim_{n\to\infty}\frac{\kappa(g^n)}{n}=\tJd(g).\]
	Moreover, when $\sG=\SL(d,\R)$ then $\kappa(g)=(\log\sigma_1(g),\dots,\log\sigma_d(g))$, where each $\sigma_i(g)$ is a singular value of $g$ such that $\sigma_1(g)\geq \dots \geq\sigma_d(g)$.

	\subsection{Anosov representations}\label{sec.AR}
	
	Let $\sG$ be a semisimple Lie group and $\sP_\pm$ be a pair of opposite parabolic subgroups of $\sG$. We consider 
	\[\cX\subset\sG/\sP_+\times\sG/\sP_-\] 
	to be the space of all pairs $(g\sP_+,g\sP_-)$ for $g\in\sG$. We observe that the left action of $\sG$ on $\sG/\sP_+\times\sG/\sP_-$ is transitive on $\cX$. Moreover, the stabilizer of the point $(\sP_+,\sP_-)\in\cX$ is $\sP_+\cap\sP_-$. It follows that $\sG/(\sP_+\cap\sP_-)\cong\cX$. Also, $\cX$ is open in $\sG/\sP_+\times\sG/\sP_-$. Therefore,
	\[\sT_{\left(g\sP_+,g\sP_-\right)}\cX = \sT_{g\sP_+}\sG/\sP_+ \oplus \sT_{g\sP_-}\sG/\sP_-.\]
	
	Let $\Gamma$ be a hyperbolic group and $\bdry$ be the boundary of $\Gamma$. There is a natural action of $\Gamma$ on $\bdry$. We consider the restriction of the diagonal action of $\Gamma$ on $\bdry^2$ to
	\[\partial_\infty\Gamma^{(2)} := \partial_\infty\Gamma\times\partial_\infty\Gamma\setminus\{(x,x) \mid x\in \partial_\infty\Gamma\}.\]
	We denote $\partial_\infty\Gamma^{(2)}\times\mathbb{R}$ by $\cflow$ and for $(x,y)\in\partial_\infty\Gamma^{(2)}$ and $s,t\in\mathbb{R}$ let
	\begin{align*}
		\phi_t: \cflow&\rightarrow \cflow\\
		(x,y,s)&\mapsto (x,y,s+t).
	\end{align*}
	Gromov \cite{Gromov} showed the existence of a proper cocompact action of $\Gamma$ on $\cflow$ which commutes with the the flow $\{\phi_t\}_{t\in\mathbb{R}}$ and which extends the diagonal action of $\Gamma$ on $\partial_\infty\Gamma^{(2)}$ (see \cite{Champetier}, \cite{Mineyev} for more details). Moreover, there exists a metric $d$ on $\cflow$ well defined only up to H\"older equivalence such that the $\Gamma$ action on $\cflow$ is isometric, the action of the flow $\phi_t$ is via Lipschitz homeomorphisms and every orbit of the flow $\{\phi_t\}_{t\in\mathbb{R}}$ gives a quasi-isometric embedding.
	We call the quotient space $\flow\defeq\Gamma\backslash\cflow$ the \textit{Gromov flow space}. $\flow$ is a compact connected metric space (see Lemma 2.3 of \cite{GT} for more details) and it admits a partition of unity (see Section 8.2 of \cite{GT} for more details). If $\gamma\in\Gamma$ is of infinite order, we define its \emph{translation length}
	\[ \ell(\gamma) = \lim_{n\to\infty} \frac{d(\gamma^nx,x)}{n}, \]
	where $x\in\cflow$ is any point. Then we have $\ell(\gamma) = \inf\{(d(y,\gamma y))\mid y\in\cflow\}$ and the infimum is realized on the \emph{axis} $\{ (\gamma_+,\gamma_-,t), t\in\R \}$ (see Theorem 60 of \cite{Mineyev} for more details).
	
	\begin{definition}
		Let $\Gamma$ be a hyperbolic group, $\sG$ be a semisimpe Lie group, $\sP_\pm$ be a pair of opposite parabolic subgroups of $\sG$ and $\rho:\Gamma\to\sG$ be an injective homomorphism. Then $\rho$ is called Anosov with respect to the pair $\sP_\pm$ if the following conditions hold:
		\begin{enumerate}
			\item There exist continuous, injective, $\rho(\Gamma)$-equivariant limit maps
			\[\xi^\pm:\bdry\rightarrow\sG/\sP_\pm\]
			such that $\xi(p)\defeq(\xi^+(p_+),\xi^-(p_-))\in\cX$ for any               $p=(p_+,p_-,t)\in\cflow$.
			\item There exist positive constants $C,c$ and a continuous collection     of $\rho(\Gamma)$-equivariant Euclidean metrics $\|\cdot\|_p$ on            $\sT_{\xi(p)}\cX$ for $p\in\cflow$ such that 
			\begin{align*}
				\|v^\pm\|_{\phi_{\pm t}p}\leqslant Ce^{-ct}\|v^\pm\|_p
			\end{align*}
			for all $v^\pm\in\sT_{\xi^\pm(p_\pm)}\sG/\sP_\pm$ and for all               $t\geqslant 0$.
		\end{enumerate}
	\end{definition}
	
	Let $P_k^+$ be the stabilizer of a standard $k$-plane and $P_k^-$ be the stabilizer of the standard complimentary $(d-k)$-plane. Then a $(P_k^+,P_k^-)$-Anosov representation is called in short a $P_k$-Anosov representation. Now we combine results obtained by Guichard--Wienhard \cite{GW2}, Kapovich--Leeb--Porti \cite{KLP}, Bochi--Potrie--Sambarino \cite{BPS}, Gu\'eritaud--Guichard--Kassel--Wienhard \cite{GGKW} and Kassel--Potrie \cite{KP} to state the following important characterization of Anosov representations (see Canary \cite{Can} and Remark 4.3 (5) of Kassel--Potrie \cite{KP} for more details):
	
	\begin{theorem}
		Let $\rho : \Gamma \to \SL(d, \R)$ be an injective homomorphism. Then the following holds:
		\begin{enumerate}
			\item $\rho$ has an uniform $k$-gap in singular values if it is a $P_k$-Anosov representation,
			\item $\rho$ is a $P_k$-Anosov representation if it has an uniform $k$-gap in singular values,
			\item $\rho$ has an uniform $k$-gap in eigenvalues if and only if it has a uniform $k$-gap in singular values provided $\Gamma$ is word hyperbolic.
		\end{enumerate}
	\end{theorem}
	
	Now we introduce the particular parabolic subgroups that we will be interested in for the rest of this article. Let $\sP^\pm\subset\sG$ be the normalizer of $\fp^{\pm,0}$, $\mathsf{Stab}_\sG(\sV^\pm)\subset\sG$ be the stabilizer of $\sV^\pm$ and $\mathsf{Stab}_\sG(\sV^{\pm,0})\subset\sG$ be the stabilizer of $\sV^{\pm,0}$.
	
	\begin{definition}\label{def.linpara}
		Any subspace of $\sV$ which is of the form $g\sV^{+,0}$ for some $g\in\sG$ is called a \textit{parabolic space} and any two parabolic spaces are called \textit{transverse} if their sum is the whole space $\sV$.
	\end{definition}

	\begin{proposition}\label{prop.hijk}
		The pair $(\sV^{+,0},\sV^{-,0})$ is a transverse pair of parabolic subspaces and for any pair of transverse parabolic spaces $(\sV_i,\sV_j)$, there exists $h_{i,j}\in\sG$ such that $(\sV_i,\sV_j)=h_{i,j}(\sV^{+,0},\sV^{-,0})$.
		
		If $\sV_i$, $\sV_j$ and $\sV_k$ are three parabolic subspaces which are mutually transverse to each other and $h_{i,j}, h_{i,k}, h_{k,j}\in\sG$ be such that
		\[(\sV^{+,0},\sV^{-,0})=h_{i,j}^{-1}(\sV_i,\sV_j)=h_{i,k}^{-1}(\sV_i,\sV_k)=h_{k,j}^{-1}(\sV_k,\sV_j),\]
		then $h_{i,j}\sV^+=h_{i,k}\sV^+$ and $h_{i,j}\sV^-=h_{k,j}\sV^-$. Moreover, we have \[\sP^\pm=\mathsf{Stab}_\sG(\sV^{\pm,0})=\mathsf{Stab}_\sG(\sV^{\pm}),\] 
		and all $Y\in\sV^0$ is preserved by any $h\in(\sP^+\cap\sP^-)$ i.e $hY=Y$. 
	\end{proposition}	
	\begin{proof}
		As $\sG$ is split, our result follows from Propositions 4.4, 4.19, 4.27, Lemma 4.21 and Example 4.22(2) of Smilga \cite{Smilga3}.
	\end{proof}
	Let $\sF^\pm$ be the flag $\{0\}\subset\sV^{\pm}\subset\sV^{\pm,0}\subset\sV$. Hence, it follows that
	\[\sP^\pm=\mathsf{Stab}_{\sG}(\sF^{\pm})\subset\mathsf{Stab}_{\SL(\sV)}(\sF^{\pm})=:\sQ^\pm.\]
	
	\begin{proposition}[Proposition 4.3 of Guichard--Wienhard \cite{GW2}]\label{prop.GW}
		Let $\tR:\sG\to\SL(\sV)$ be a faithful irreducible representation. Then $\rho:\Gamma\to\sG$ is Anosov with respect to $(\sP^+,\sP^-)$ if and only if $\tR\circ\rho$ is Anosov in $\SL(\sV)$ with respect to $(\sQ^+,\sQ^-)$. 
	\end{proposition}
	
	\begin{definition}\label{def.ano}
		We denote the space of all representations $\rho:\Gamma\to\sG$ which are Anosov with respect to $(\sP^+,\sP^-)$ by $\ano(\Gamma,\sG,\tR)$.
	\end{definition}
	
	\section{Affine deformations}\label{sec.2}
	\subsection{Neutral and neutralized sections}\label{sec.NS}
	
	We consider the space of transverse parabolic subspaces and using Proposition \ref{prop.hijk} deduce that the space of transverse parabolic subspaces can be identified with $\sG/(\sP^+\cap\sP^-)$.
	
	\begin{notation}
		Let $\pi_0:\sV\to\sV^0$ be the projection map with respect to the decomposition $\sV=\sV^+\oplus\sV^0\oplus\sV^-$.
	\end{notation}
	
	\begin{definition}\label{def.nu}
		Suppose $\nu:\sG\to\sHom(\sV,\sV)$ be such that $\nu(g)=g\pi_0$ for all $g\in\sG$. Then $\nu$ induces the following map:
		\[\nu:\sG/(\sP^+\cap\sP^-)\to\sHom(\sV,\sV),\]
		which we call the generalized neutral map. Moreover, for each $Y\in\sV^0$ we define $\nu_Y$, the neutral map with respect to $Y$, as follows: $\nu_Y([g])=\nu([g])Y$ for all $g\in\sG$. 
	\end{definition}
	\begin{remark}\label{rem.pm}
		Suppose $\sW$ is a parabolic space. Then by Proposition \ref{prop.hijk} there exists unique $\sW^\pm\subset \sW$ such that $\sW=g\sV^{\pm,0}$ if and only if $\sW^+=g\sV^\pm$.
	\end{remark}
	
	\begin{lemma}\label{lem.nu2}
		Suppose $Y\in\sV^0$ and $\nu_Y$ is a neutral map. Then for any three mutually transverse parabolic subspaces $\sV_i,\sV_j,\sV_k$ we have
		\begin{align*}
			[\nu_Y(\sV_i,\sV_k)-\nu_Y(\sV_i,\sV_j)]\in\sV_i^+ \text{ and }
			[\nu_Y(\sV_k,\sV_j)-\nu_Y(\sV_i,\sV_j)]\in\sV_j^-.
		\end{align*}
	\end{lemma}
	\begin{proof}
		We recall that $\sV^{\pm,0}$ are respectively the attracting and repelling fixed points for the action of $g:=\exp(X_0)$ on $\sG/(\sP^+\cap\sP^-)$. Hence, for any $\sW$ transverse to both $\sV^{\pm,0}$ we have $\lim_{n\to\infty}g^{-n}\sW=\sV^{-,0}$.	Suppose $(\sV_i,\sV_j)=h_{i,j}(\sV^{+,0},\sV^{-,0})$. Then $h_{i,j}^{-1}\sV_k$ is transverse to both $\sV^{\pm,0}$. Moreover, $\nu_Y$ is a continuous map. Hence,
		\[\lim_{n\to\infty}h_{i,j}g^{-n}h_{i,j}^{-1}\nu_Y(V_i,V_k)=\nu_Y(V_i,V_j)\]
		and our result follows.
	\end{proof}
	
	Suppose $\rho\in\ano(\Gamma,\sG,\tR)$ with limit map $\xi_\rho$. We consider the bundle $\cflow\times\sV$ over $\cflow$ and abuse notation to define the flow $\phi_t(p,Y):=(\phi_tp,Y)$. Moreover, for the following action of $\Gamma$: $\gamma(x,Y):=(\gamma x, \rho(\gamma)Y)$, we consider the quotient bundle 
	\[\cV:=\Gamma\backslash(\cflow\times\sV)\text{ over }\flow.\] 
	The flow $\phi$ descends to a flow on $\cV$ which we again denote by $\phi$. 
	We observe that for $p\in\cflow$ there exists some $g_p\in\sG$ such that $\xi^\pm_\rho(p_\pm)=g_p\sQ^\pm$. Hence, the limit map $\xi_\rho$ induces a splitting of the bundle $\cV$ as follows:
	\[\cV:=\cV^+\oplus\cV^0\oplus\cV^-,\]
	where the fiber over $p\in\cflow$ is denoted by $(p, \sV^+_{p_+}, \sV^0_p, \sV^-_{p_-})$. We observe that $(p, \sV^+_{p_+}, \sV^0_p, \sV^-_{p_-})=(p, g_p\sV^+, g_p\sV^0, g_p\sV^-)$ and $\sV^\pm_{p_\pm}\oplus\sV^0_p=\sV^{\pm,0}_{p_\pm}$.
	
	We observe that for any $Y\in\sV^0$ we have $\nu_Y(hg)=h\nu_Y(g)$ for all $g,h\in\sG$. Hence, it follows that $\rho(\gamma)\nu_Y\circ\xi(\gamma_+,\gamma_-) =\nu_Y\circ\xi(\gamma_+,\gamma_-)$.	We abuse notation and denote $\nu_Y\circ\xi$ by $\nu_Y$. These maps induce the following sections which we again denote using the same notation: 
	\[\nu_Y:\flow\to\cV^0.\] 
	We call these the \textit{neutral sections}. We observe that the neutral sections induced by the standard neutral maps spans the space of sections of $\cV^0$. 
	
	\begin{remark}
		Suppose $h\in\sG$ satisfy $h^{-1}\rho(\gamma)_\tth h=\exp(\tJd_g)$. Then using Propositions 4.16 and 4.21 of Smilga \cite{Smilga3} we obtain that $\nu(\gamma_+,\gamma_-)=h\pi_0$.
	\end{remark}
	
	Now we consider the bundle $\cflow\times\sV \text{ over } \cflow$	and introduce the affine analogues of the bundle $\cV$ over $\flow$. Suppose $u:\Gamma\to\sV$ be a cocycle at $\rho$ i.e. $(\rho,u):\Gamma\to\sG\ltimes\sV$ is an injective homomorphism.
	Then for the following action of $\Gamma$: $\gamma(x,Y):=(\gamma x, (\rho(\gamma),u(\gamma))Y)$, we consider the quotient bundle 
	\[\acV:=\Gamma\backslash(\cflow\times\sV)\text{ over }\flow.\] 
	The flow $\phi$ descends to a flow on $\acV$ which we again denote by $\phi$.
	We can use a partition of unity argument to construct H\"older continuous sections $\sigma: \flow \to \acV$ which are differentiable along the flow lines (see \cite{GT} for more details on this). Let $\Tilde{\sigma}$ be the lift of $\sigma$. We define
	\[{\nabla}_\phi\Tilde{\sigma} (p):=\left.\frac{\partial}{\partial t}\right|_{t=0}\phi_{-t}\Tilde{\sigma}(\phi_tp).\]
	Let $\nabla^\pm_\phi\tilde{\sigma} (p)$ and $\nabla^0_\phi\tilde{\sigma} (p)$ denote the projections of ${\nabla}_\phi\tilde{\sigma} (p)$ respectively on $\sV^\pm_{p_\pm}$ and $\sV^0_p$. Hence, we obtain sections $\nabla_\phi\sigma:\flow\to\cV$, $\nabla_\phi^\pm\sigma:\flow\to\cV^\pm$ and $\nabla_\phi^0\sigma:\flow\to\cV^0$ such that $\nabla_\phi\sigma=\nabla^+_\phi\sigma+\nabla^0_\phi\sigma+\nabla^-_\phi\sigma$.
	
	\begin{lemma}\label{lem.neutralized1}
		Let $\sigma: \flow \to \acV$ be a H\"older continuous section which is differentiable along the flow lines. Then there exist $\sigma_0: \flow \to \acV$ such that
		\[\nabla_\phi\sigma_0=\nabla^0_\phi\sigma .\]
	\end{lemma}
	
	\begin{proof}
		We observe that $\nabla_\phi \sigma$ and $\nabla^\pm_\phi \sigma$ are all $(\rho,u)(\Gamma)$-equivariant. We use the Anosov property and argue as in Proposition 5.2 of Ghosh--Treib \cite{GT} to obtain that the following two integrals are well defined:
		\[\int_0^\infty \phi_t\left( \nabla^-_\phi \sigma(\phi_{- t}p)\right) dt \text{ , } \int_0^\infty \phi_{-t}\left( \nabla^+_\phi \sigma(\phi_{t}p) \right) dt.\]
		Hence, the following map is well defined:
		\begin{align*}
			\sigma_0(p):= \sigma(p) &- \int_0^\infty \phi_t\left( \nabla^-_\phi \sigma(\phi_{- t}p)\right) dt + \int_0^\infty  \phi_{-t}\left( \nabla^+_\phi \sigma(\phi_{t}p) \right) dt.
		\end{align*}
		Now a routine computation implies that
		\[\nabla_\phi \sigma_0(p) = \nabla_\phi \sigma(p)-\nabla^-_\phi \sigma(p) -\nabla^+_\phi \sigma(p). \]
		Our result follows.
	\end{proof}
	\begin{definition}
		Any section ${\sigma}:\flow\to\acV$ for which $\nabla_\phi\sigma=\nabla^0_\phi\sigma$ is called a neutralized section. 
	\end{definition}

	\begin{lemma}\label{lem.neutralized2}
		Suppose $\sigma: \flow \to \acV$ is a neutralized section and $\tilde{\sigma}:\cflow\to\sV$ be the map induced by the lift of $\sigma$. Then  
		\[\tilde{\sigma}(a,b,0)-\tilde{\sigma}(a,c,0)\in\sV^{+,0}_a \text{ and } \tilde{\sigma}(a,b,0)-\tilde{\sigma}(d,b,0)\in\sV^{-,0}_b.\]
	\end{lemma}
	\begin{proof}
		As $\tilde{\sigma}$ is $(\rho,u)(\Gamma)$-equivariant, there exist $t_n,s_n\in\R$ such that 
		\[\rho(\gamma)^n[\tilde{\sigma}(a,\gamma_-,0)-\tilde{\sigma}(\gamma_+,\gamma_-,0)]=[\tilde{\sigma}(\gamma^na,\gamma_-,t_n)-\tilde{\sigma}(\gamma_+,\gamma_-,s_n)].\] 
		As $\sigma$ is neutralized, there exists some constant $C>0$ such that 
		\begin{align*}
			&\|\tilde{\sigma}(\gamma^na,\gamma_-,t_n)-\tilde{\sigma}(\gamma^na,\gamma_-,0)\|\leq Ct_n,\\
			&\|\tilde{\sigma}(\gamma_+,\gamma_-,s_n)-\tilde{\sigma}(\gamma_+,\gamma_-,0)\|\leq Cs_n.
		\end{align*}
		Moreover, as $[\tilde{\sigma}(\gamma^na,\gamma_-,0)-\tilde{\sigma}(\gamma_+,\gamma_-,0)]$
		stays bounded and $t_n/n$ and $s_n/n$ stays bounded we obtain that 
		\[\frac{1}{n}\rho(\gamma)^n[\tilde{\sigma}(a,\gamma_-,0)-\tilde{\sigma}(\gamma_+,\gamma_-,0)]\]
		stays bounded. It follows that $[\tilde{\sigma}(a,\gamma_-,0)-\tilde{\sigma}(\gamma_+,\gamma_-,0)]\in \sV^{-,0}_{\gamma_-}$. Now using continuity we deduce that $[\tilde{\sigma}(a,b,0)-\tilde{\sigma}(d,b,0)]\in\sV^{-,0}_b$. Similarly, we obtain that $[\tilde{\sigma}(a,b,0)-\tilde{\sigma}(a,c,0)]\in\sV^{+,0}_a$.
	\end{proof}

	\subsection{Partial affine Anosov representations}\label{sec.PAA}
	The essential idea behind this section is the following simple observation: 	
	
	\begin{lemma}\label{lem.stableaffineplane}
		Let $g\in{\sG}$ be of the same type as $X_0$ with respect to $\tR$ and $\sV^0_g$ be the unit eigenspace of $g$. Then for any $Y\in{\sV}$ there exists a unique affine subspace $\sA_{g,Y}$ of $\sV$ which is parallel to  such that $(g,Y)\sA_{g,Y}=\sA_{g,Y}$.
	\end{lemma}
	\begin{proof}
		We consider some Euclidean norm $\|\cdot\|$ on $\sV$. As $g\in{\sG}$ is of the same type as $X_0$ with respect to $\tR$, we observe that there exists linear subspaces $\sV^\pm_g$ and $\sV^0_g$ of $\sV$ which are stabilized by $g$. Moreover, for any $Y^\pm\in\sV^\pm$, $Y^0\in\sV^0$ we have $\|gY^+\|>\lambda\|Y^+\|$, $\|gY^-\|<\lambda^{-1}\|Y^+\|$ for some $\lambda>1$ and $gY^0=Y^0$. Suppose $Y\in\sV$. We denote its projections on $\sV_g^\pm$ (resp. $\sV_g^0$) by $Y^\pm$ (resp. $Y^0$) and observe that
		\begin{align*}
			\left\|\left(\sum_{n=0}^\infty g^n\right)Y^-\right\|&\leq\left(\sum_{n=0}^\infty \|g^nY^-\|\right)\leq\left(\sum_{n=0}^\infty \lambda^{-n}\right)\|Y^-\|<\infty,\\
			\left\|\left(\sum_{n=0}^\infty g^{-n}\right)Y^+\right\|&\leq\left(\sum_{n=0}^\infty \|g^{-n}Y^+\|\right)\leq\left(\sum_{n=0}^\infty \lambda^{-n}\right)\|Y^+\|<\infty.
		\end{align*}
		Hence, the following expression is well defined:
		\[\sA_{g,Y}:=\sV_g^0+\left(\sum_{n=0}^\infty g^n\right)Y^--g^{-1}\left(\sum_{n=0}^\infty g^{-n}\right)Y^+.\]
		We observe that $(g,Y)\sA_{g,Y}=\sA_{g,Y}$. 
		
		Suppose $\sA_1$ and $\sA_2$ are two such subspaces which are parallel to $\sV^0_g$. Then $\sA_1=\sA_2+Z$ for some $Z\in\sV$ and $gZ=Z$. Therefore, $Z\in\sV^0_g$ and it implies that $\sA_1=\sA_2+Z=\sA_2$, proving uniqueness.
	\end{proof}
	
	Now we introduce appropriate notions to generalize the above result and frame it in the context of Anosov representations.
	
	\begin{definition}\label{def.affpara}
		Suppose $(\sV_i,\sV_j)$ is a transverse pair of parabolic subspaces of $\sV$ and $(\sA_i,\sA_j)$ are two affine subspaces of $\sV$ such that $\sV_i$ is respectively parallel to $\sA_i$. We call such spaces \textit{affine parabolic spaces} and two affine parabolic spaces $\sA_i$ and $\sA_j$ to be \textit{transverse} to each other if and only if their corresponding vector spaces $\sV_i$ and $\sV_j$ are transverse to each other.
	\end{definition}	
	
	We simplify our notations and denote $\SL(\sV)\ltimes\sV$ by $\ASL$ and $\sG\ltimes_\tR\sV$ by $\AG$. Moreover, we denote the stabilizer of the affine space $\sV^{\pm,0}$ under the action of $\AG$ (resp. $\ASL$) on $\sV$ by $\asP^\pm$ (resp. $\asQ^\pm$). We call these subgroups a pair of \textit{opposite pseudo-parabolic} subgroups respectively of $\AG$ and $\ASL$. Let $\cY_\tR\subset\AG/\asP^+\times\AG/\asP^-$ be the space of all pairs of transverse affine parabolic subspaces of $\sV$. We use Proposition \ref{prop.hijk} and obtain that the left action of $\AG$ on $\AG/\asP^+\times\AG/\asP^-$ is transitive on $\cY_\tR$. Hence, we have 
	\[\cY_\tR=\{\left((g,X)\sV^{+,0},(g,X)\sV^{-,0}\right)\mid (g,X)\in\AG\}\cong\AG/(\asP^+\cap\asP^-).\]
	Transversality imply that $\cY_\tR$ is open in $\AG/\asP^+\times\AG/\asP^-$. Therefore,
	\[\sT_{(g,X)\left(\asP^+,\asP^-\right)}\cY_\tR = \sT_{(g,X)\asP^+}(\AG/\asP^+) \oplus \sT_{(g,X)\asP^-}(\AG/\asP^-).\]
	We again use Proposition \ref{prop.hijk} and identify $\cY_\tR$  naturally with a subset of
	\[\cY:=\{\left((g\sV^+,(g,X)\sV^{+,0}),(g\sV^-,(g,X)\sV^{-,0})\right)\mid (g,X)\in\ASL\}.\]	
	Similarly, we can identify $\cY$ with an open subset of $\ASL/\asQ^+\times\ASL/\asQ^-$ to obtain $\cY\cong\ASL/(\asQ^+\cap\asQ^-)$. Hence,
	\[\sT_{(g,X)\left(\asQ^+,\asQ^-\right)}\cY = \sT_{(g,X)\asQ^+}(\ASL/\asQ^+) \oplus \sT_{(g,X)\asQ^-}(\ASL/\asQ^-).\]
	
	\begin{definition}
		Let $\Gamma$ be a hyperbolic group, $\sG$ be a split semisimpe Lie group and $\tR:\sG\to\SL(\sV)$ be a faithful irreducible representation. Let $\asP^\pm$ (resp. $\asQ^\pm$) be a pair of opposite pseudo-parabolic subgroups of $\AG$ (resp. $\ASL$) and $(\rho,u):\Gamma\to\AG\subset\ASL$ be an injective homomorphism. Then $(\rho,u)$ is called partial affine Anosov in $\AG$ (resp. $\ASL$) with respect to the pair $\asP^\pm$ (resp. $\asQ^\pm$) if the following conditions hold:
		\begin{enumerate}
			\item There exist continuous, injective, $(\rho,u)(\Gamma)$-equivariant limit maps
			\[\xi^\pm:\bdry\rightarrow\AG/\asP^\pm\subset\ASL/\asQ^\pm\]
			such that $\xi(p)\defeq(\xi^+(p_+),\xi^-(p_-))\in\cY$ for any                     $p=(p_+,p_-,t)\in\cflow$.
			\item There exist positive constants $C,c$ and a continuous collection of         $\rho(\Gamma)$-equivariant Euclidean metrics $\|\cdot\|_p$ on                     $\sT_{\xi(p)}\cY_\tR$ (resp. $\sT_{\xi(p)}\cY$) for $p\in\cflow$ such that 
			\begin{align*}
				\|v^\pm\|_{\phi_{\pm t}p}\leqslant Ce^{-ct}\|v^\pm\|_p
			\end{align*}
			for all $v^\pm\in\sT_{\xi^\pm(p_\pm)}\AG/\asP^\pm$ (resp. $v^\pm\in\sT_{\xi^\pm(p_\pm)}\ASL/\asQ^\pm$) and for all $t\geqslant 0$.
		\end{enumerate}
	\end{definition}
	
	We note that similar notion as that of a partial affine Anosov representation has also been independently introduced by Kassel--Smilga in \cite{KaSm}.
	
	\begin{remark}
		Suppose $(\rho,u)$ is partially affine Anosov. We observe that $\SL(\sV)/\sQ^\pm\subset\ASL/\asQ^\pm$. Hence, we deduce that $\rho\in\ano(\Gamma,\sG,\tR)$.
	\end{remark}
	
	\begin{proposition}\label{prop.partialaffineano}
		Let $\tR:\sG\to\SL(\sV)$ be an injective homomorphism, $\rho\in{\ano(\Gamma,\sG,\tR)}$ and $u$ be a $\sV$ valued cocycle with respect to $\rho$. Then $(\rho,u)$ is a {partial affine Anosov} representation in $\AG$ (resp. $\ASL$) with respect to $\asP^\pm$ (resp. $\asQ^\pm$).
	\end{proposition}
	\begin{proof}
		Let $\sigma: \flow \to \acV$ be a H\"older continuous section which is differentiable along the flow lines of $\flow$. We use Lemma \ref{lem.neutralized1} to obtain a neutralized section
		$\sigma_u: \flow \to \acV$ such that $\nabla_\phi\sigma_u=\nabla^0_\phi\sigma$. Hence, we obtain that $\nabla_\phi\sigma_u:\flow\to\cV^0$ is a section. Suppose $\tilde{\sigma}_u$ is the lift of $\sigma_u$. Hence, it induces $\tilde{\sigma}_u:\cflow\to\sV$. We use Lemma \ref{lem.neutralized1} to deduce that $\left(\tilde{\sigma}_u(\phi_tp)-\tilde{\sigma}_u(p)\right)\in \sV^0_p$. Also, $\tilde{\sigma}_u(p)+\sV^{\pm,0}_{p_\pm}=(g_p,\tilde{\sigma}_u(p))\sV^{\pm,0}$ and we define,
		\[\xi_u(p):=\left(\tilde{\sigma}_u(p)+\sV^{+,0}_{p_+}, \tilde{\sigma}_u(p)+\sV^{-,0}_{p_-}\right)\]
		and notice that it is $(\rho,u)(\Gamma)$-equivariant and invariant under the flow. Hence, we obtain a limit map \[\xi_u:\bdry^{(2)}\to\AG/\asP^+\times\AG/\asP^-\subset\ASL/\asQ^+\times\ASL/\asQ^-\] 
		such that $\xi_u(\bdry^{(2)})\subset\cY_\tR\subset\cY$. Moreover, we observe that
		\[\sT_{\xi_u^\pm(p_\pm)}\ASL/\asQ^\pm=\sT_{\xi^\pm(p_\pm)}\SL(\sV)/\sQ^\pm\oplus g_p\sV^\mp.\]
		As $\rho\in\ano(\Gamma,\sG,\tR)$, we use Proposition \ref{prop.GW} and argue like Lemma 3.2 and Corollary 3.3 of Ghosh--Treib \cite{GT} to deduce that $(\rho,u)$ is partially affine Anosov in $\ASL$ with respect to $\asQ^\pm$. Moreover, as $\xi_u(\bdry^{(2)})\subset\cY_\tR\subset\cY$, it directly follows that $(\rho,u)$ is partially affine Anosov in $\AG$ with respect to $\asP^\pm$ too.
	\end{proof}
	
	\section{Affine invariants}\label{sec.3}

	\subsection{Affine cross ratios and triple ratios}\label{sec.ACTR}
	Suppose $\{\sV_i\}_{i=1}^{4}$ are four mutually transverse pair of parabolic subspaces of $\sV$ and $\{\sA_i\}_{i=1}^{4}$ are four affine subspaces in $\sV$ such that $\sV_i$ is respectively parallel to $\sA_i$. We recall Remark \ref{rem.pm} and consider the subspaces $\sV_i^\pm\subset\sV_i$. Hence,
	\[\sV=\sV_i^+\oplus(\sV_i\cap\sV_j)\oplus\sV_j^-.\]
	\begin{notation}
		Let $\pi_{i,j}:\sV\to\sV_i\cap\sV_j$ be the projection map with respect to the decomposition $\sV=\sV_i^+\oplus(\sV_i\cap\sV_j)\oplus\sV_j^-$.
		Suppose $h_{i,j}\in\sG$ be such that $(\sV_i,\sV_j)=h_{i,j}(\sV^{+,0},\sV^{-,0})$. Then $\pi_{i,j}=h_{i,j}\pi_0h_{i,j}^{-1}$.
	\end{notation}
	
	\begin{definition}
		Suppose $\nu^*:\sG\to\sHom(\sV,\sV^0)$ be such that $\nu^*(g)=\pi_0g^{-1}$ for all $g\in\sG$. Then $\nu^*$ induces the following map (see Proposition \ref{prop.hijk}):
		\[\nu^*:\sG/(\sP^+\cap\sP^-)\to\sHom(\sV,\sV^0),\]
		which we call the co-neutral map.
	\end{definition}
	Moreover, we recall the map $\nu$ from Definition \ref{def.nu} and for notational simplicity we denote $\nu(\sV_i,\sV_j)$ by $\nu_{i,j}$ and $\nu^*(\sV_i,\sV_j)$ by $\nu_{i,j}^*$.
	
	\begin{lemma}\label{lem.nu*}
		Suppose $\nu$ is the generalized neutral map and $\nu^*$ is the co-neutral map. Then $\ker(\nu^*_{i,j})=\sV_i^+\oplus\sV_j^-$ and $\mathrm{img}(\nu_{i,j})=\sV_i\cap\sV_j$. Moreover, they satisfy the following properties:
		\begin{enumerate}
			\item $\nu_{i,j}\nu^*_{i,j}=\pi_{i,j}$ and $\nu^*_{i,j}\nu_{i,j}=\pi_0$,
			\item $\nu^*_{i,j}=\nu^*_{i,k}\pi_{i,j}$ and $\nu_{i,j}=\pi_{i,j}\nu_{i,k}$,
			\item $\nu^*_{i,j}=\nu^*_{k,j}\pi_{i,j}$ and $\nu_{i,j}=\pi_{i,j}\nu_{k,j}$.
		\end{enumerate}		
	\end{lemma}
	\begin{proof}
		Follows directly from Proposition \ref{prop.hijk} and Lemma \ref{lem.nu2}.
	\end{proof}

	\begin{definition}\label{def.affcr}
		Let $X_{i,j}$ be a point in $\sA_i\cap \sA_j$ for $i,j\in\{1,2,3,4\}$. Then the {affine cross ratio} of the four mutually transverse affine parabolic spaces $\{\sA_i\}_{i=1}^4$ is defined as follows:
		\[\beta_{1,2,3,4}=\beta(\sA_1,\sA_2,\sA_3,\sA_4)\defeq [{{\nu^*_{1,4}-\nu^*_{2,3}}}](X_{1,3}-X_{2,4}).\]
	\end{definition}
	
	We use Lemma \ref{lem.nu*} to observe that the above definition is well defined i.e. $\beta_{1,2,3,4}$ is independent of the choice of $X_{1,3}$ and $X_{2,4}$.
	
	\begin{lemma}\label{lem.altmarg}
		Let $X_i$ be a point inside $\sA_i$ for $1\leq i\leq4$. Then
		\[\beta_{1,2,3,4}= [\nu^*_{1,4}-\nu^*_{1,3}] X_1+[\nu^*_{2,3}-\nu^*_{2,4}] X_2-[\nu^*_{2,3}-\nu^*_{1,3}]X_3 -[\nu^*_{1,4}-\nu^*_{2,4}]X_4.\]
	\end{lemma}
	\begin{proof}
		Suppose $X_{i,j}\in \sA_i\cap \sA_j$ and let $X_{i,j}^+$ and $X_{i,j}^-$ respectively be the projection of $X_{i,j}$ on $\sV_i^+$ and $\sV_j^-$ with respect to the decomposition 
		\[\sV=\sV_i^+\oplus(\sV_i\cap \sV_j)\oplus \sV_j^-.\]
		We observe that as $X_{i,j}$ varies along $\sA_i\cap \sA_j$ the projection $X_{i,j}^\pm$ stays fixed. Moreover, $X_{i,j}^-+\sV_i=\sA_i$ and $X_{i,j}^++\sV_j=\sA_j$. Hence, using Lemma \ref{lem.nu*} we obtain that $\nu^*_{i,j}(X_{i,j}^\pm)=0$ and deduce that
		\begin{align*}
			\beta_{1,2,3,4}= 
			&\ [\nu^*_{1,4}-\nu^*_{2,3}]( X_{1,3}^++ X_{1,3}^--X_{2,4}^+-X_{2,4}^-)\\
			=&\ \nu^*_{1,4}(X_{1,3}^-)+\nu^*_{2,3}
			(X_{2,4}^-)-\nu^*_{2,3}( X_{1,3}^+) -\nu^*_{1,4}( X_{2,4}^+)\\
			=&\ [\nu^*_{1,4}-\nu^*_{1,3}](X_{1,3}^-)+
			[\nu^*_{2,3}-\nu^*_{2,4}](X_{2,4}^-)-[\nu^*_{2,3}-\nu^*_{1,3}](X_{1,3}^+)\\ 
			&-[\nu^*_{1,4}-\nu^*_{2,4}](X_{2,4}^+).
		\end{align*}
		Finally, again using Lemma \ref{lem.nu*} our result follows.
	\end{proof}
	
	\begin{definition}\label{def.afftr}
		Suppose $X_i$ is a point inside $\sA_i$ for $i\in\{2,3,4\}$. Then the {affine triple ratio} of three mutually transverse affine parabolic spaces $\{\sA_i\}_{i=2}^4$ is defined as follows:
		\begin{align*}
			\delta_{2,3,4}= \delta(\sA_2,\sA_3,\sA_4)
			\defeq&\  [\nu^*_{2,3}+\nu^*_{3,2}](X_2-X_3)+          [\nu^*_{3,4}+\nu^*_{4,3}](X_3-X_4)\\
			&\ +[\nu^*_{4,2}+\nu^*_{2,4}](X_4-X_2).
		\end{align*}
	\end{definition}
	
	We use Lemma \ref{lem.nu*} to observe that the above definition is well-defined and observe that $\delta_{2,3,4}=\delta_{3,4,2}=\delta_{4,2,3}=-\delta_{3,2,4}=-\delta_{2,4,3}=-\delta_{4,3,2}$. Moreover, for any $(g,Y)\in\AG$ we have
	\[\delta(\sA_2,\sA_3,\sA_4)=\delta((g,Y) \sA_2, (g,Y) \sA_3, (g,Y) \sA_4).\]
	Finally, we note that $\omega_0\delta=\delta$.
	
	\begin{proposition}\label{prop.cr}
		Suppose $\beta$ is the affine cross ratio and $\delta$ is the affine triple ratio. Then for any five affine parabolic spaces $\sA_*,\{\sA_i\}_{i=1}^{4}$ which are mutually transverse to each other and for any $(g,Y)\in\AG$ the following identities hold:
		\begin{enumerate}
			\item $\beta((g,Y)\sA_1,(g,Y)\sA_2,(g,Y)\sA_3,(g,Y)\sA_4)=\beta(\sA_1,\sA_2,\sA_3,\sA_4)$,
			\item \label{identity:afcros} $\beta_{1,2,3,4}=\beta_{2,1,4,3}=-\omega_0\beta_{3,4,1,2}=-\omega_0\beta_{4,3,2,1}$,
			\item $\beta_{1,2,3,4}+\beta_{1,2,4,3}=0$,
			\item $\beta_{1,*,3,4}+\beta_{*,2,3,4}=\beta_{1,2,3,4}$,
			\item $\delta_{2,3,4}=\beta_{*,2,3,4}+\beta_{*,3,4,2}+\beta_{*,4,2,3}$
		\end{enumerate}
	\end{proposition}
	\begin{proof}
		We use the definition of $\beta$ to deduce that for all $(g,Y)\in\AG$,
		\[\beta((g,Y)\sA_1,(g,Y)\sA_2,(g,Y)\sA_3,(g,Y)\sA_4)=\beta(\sA_1,\sA_2,\sA_3,\sA_4).\]
		Moreover, exploiting the symmetries in the definition of $\beta$ we obtain the identity \ref{identity:afcros}. 
		Now interchanging $A_3$ and $A_4$ in Lemma \ref{lem.altmarg} and adding them up we obtain that $\beta_{1,2,3,4}+\beta_{1,2,4,3}=0$. 
		Suppose $\sA_*$ is another affine parabolic space which is mutually transverse with the other affine parabolic spaces $\{\sA_i\}_{i=1}^4$. Now a direct computation using Lemma \ref{lem.altmarg} gives us $\beta_{1,*,3,4}+\beta_{*,2,3,4}=\beta_{1,2,3,4}$. Finally, we cyclically permute $\sA_2,\sA_3,\sA_4$ in Lemma \ref{lem.altmarg} and add them up to conclude our result.
	\end{proof}
	
	\subsection{Margulis invariants}\label{sec.MI}
	
	In this subsection we introduce Margulis invariants and relate them with affine cross ratios and in special cases (i.e. when the representation $\tR$ is the adjoint representation) with infinitesimal Jordan projections.
	
	\begin{definition}\label{def.marginv}
		Suppose $(g,Y)\in\AG$ is such that $\tJd_g$ is of type $X_0$, let $g_\tth$ be the hyperbolic part of $g$ with respect to the Jordan decomposition. Also, let $h\in\sG$ be such that $g_\tth=h\texp(\tJd_g)h^{-1}$. Then the \textit{Margulis} invariant of $(g,Y)$, is defined as follows:
		\[\tM(g,Y):=\pi_0({h}^{-1}Y).\]
	\end{definition}
	
	\begin{theorem}\label{thm.alphabeta}
		Suppose $(g,Y)\in\AG$ be such that its action on the space of affine parabolic subspaces has an attracting (resp. repelling) fixed point $\sA_+$ (resp. $\sA_-$) and $\sA_\pm$ are transverse to each other. Then for any affine parabolic space $\sA$ which is transverse to both $\sA_\pm$ the following holds:
		\[\beta(\sA_+,\sA_-,(g,Y) \sA, A)=\tM(g,Y)+\tM((g,Y)^{-1})={(1-\omega_0)\tM(g,Y)}.\]
	\end{theorem}
	\begin{proof}
		Suppose $X_\pm,X,X_{g,Y}$ are any four points respectively in $\sA_\pm$, $\sA$ and $(g,Y)\sA$. Suppose $\sW_\pm$ and $\sW$ be the corresponding vector subspaces parallel to $\sA_\pm$ and $\sA$. We simplify our notations and denote $\nu^*(\sW_\pm,\sW)$ by $\nu^*_{\pm,w}$, $\nu^*(\sW_\pm,g\sW)$ by $\nu^*_{\pm,gw}$ and $\nu^*(\sW_\pm,\sW_\mp)$ by $\nu^*_{\pm,\mp}$.
		
		Suppose $(\sW_\pm,\sW)=[h_{\pm,w}]\in\sG/(\sP^+\cap\sP^-)$. Then $(\sW_\pm,g\sW)=[gh_{\pm,w}]$. Hence, we use Lemma \ref{lem.nu*} and for any $h_{\pm,gw}\in\sG$ which satisfy $(\sW_\pm,g\sW)=[h_{\pm,gw}]$ we obtain that 
		\[ \nu^*_{\pm,gw}(X_\pm)=\nu^*_{\pm,w}(g^{-1}X_\pm).\] 
		It follows that
		\begin{align}\label{identity:1}
			\notag{[\nu^*_{\pm,w}-\nu^*_{\pm,gw}]} (X_\pm)&={\nu^*_{\pm,w}} (X_\pm-{g^{-1}}X_\pm)\\
			&={\nu^*_{\pm,w}}( X_\pm-(g,Y)^{-1}X_\pm-g^{-1}Y).
		\end{align}
		As $X_{g,Y}=(g,Y)X^\prime$ for some $X^\prime\in \sA$ and $(X^\prime-X)\in \sW$, using Lemma \ref{lem.nu*} we obtain that $[\nu^*_{-,w}-\nu^*_{+,w}](X^\prime-X)=0$ and
		\[[\nu^*_{-,gw}-\nu^*_{+,gw}](X_{g,Y})=
		[\nu^*_{-,w}-\nu^*_{+,w}](g^{-1}Y+X).\]
		Now replacing $[\nu^*_{-,gw}-\nu^*_{+,gw}](X_{g,Y})+
		[\nu^*_{+,w}-\nu^*_{-,w}](X)$ in Lemma \ref{lem.altmarg} by $[\nu^*_{-,w}-\nu^*_{+,w}](g^{-1}Y)$ and using identity \ref{identity:1} we deduce that
		\[\beta(\sA_+,\sA_-,(g,Y) \sA, \sA)={\nu^*_{+,w}}(X_+-(g,Y)^{-1}X_+)-{\nu^*_{-,w}}( X_--(g,Y)^{-1}X_-).\]
		We recall that $(g,Y)$ fixes $\sA_\pm$ and hence 
		$(X_\pm-(g,Y)^{-1}X_\pm)\in \sW_\pm$. Hence,  
		${[\nu^*_{\pm,w}-\nu^*_{\pm,\mp}]}( X_\pm-(g,Y)^{-1}X_\pm)=0$ and we deduce that
		\begin{align*}
			\beta(\sA_+,\sA_-,(g,Y) \sA, \sA)&= {\nu^*_{+,-}}(X_+-(g,Y)^{-1}X_+)- {\nu^*_{-,+}}(X_--(g,Y)^{-1}X_-)\\
			&= \tM(g,Y)+\tM((g,Y)^{-1}).
		\end{align*}
		Finally, using Proposition 8.1 of Smilga \cite{Smilga3} we conclude our result.
	\end{proof}
	
	\begin{lemma}\label{lem.conjugate}
		Suppose $g\in\sG$ is of type $X_0$. Then there exists $(h,X)\in\AG$ and $m\in\sM$ such that 
		\[(h,X)^{-1}(g,Y)(h,X)=(m\exp(\tJd_g),\tM(g,Y)).\]
	\end{lemma}
	\begin{proof}
		As $g\in{\sG}$ is of type $X_0$, we observe that there exists linear subspaces $\sV^\pm_g$ and $\sV^0_g$ of $\sV$ which are stabilized by $g$ such that the action of $g$ expands elements of $\sV^+_g$, contracts elements of $\sV^-_g$ and $gZ=Z$ for all $Z\in\sV_g^0$. Suppose $Y\in\sV$. We denote its projections on $\sV_g^\pm$ and $\sV_g^0$ respectively by $Y^\pm$ and $Y^0$. Hence, the following expressions are well defined:
		\[ X^+:=-g^{-1}\left(\sum_{n=0}^\infty g^{-n}\right)Y^+ \text{ and }
		X^-:=\left(\sum_{n=0}^\infty g^n\right)Y^-.\]
		We observe that $(X^\pm-gX^\pm)=Y^\pm$. Finally, we choose $X:=X^++X^-$ and choose $h$ such that $g=hm\exp(\tJd(g))h^{-1}$. Hence, we have
		\begin{align*}
			(h,X)^{-1}(g,Y)(h,X)&=(h^{-1}gh,h^{-1}Y-h^{-1}X-h^{-1}gX)\\
			&=(m\exp(\tJd_g),h^{-1}Y^0)=(m\exp(\tJd_g),\tM(g,Y)),
		\end{align*}
		and we conclude our result.
	\end{proof}

	\begin{proposition}\label{prop.derivative}
		Let $\sG$ be a split semisimple Lie group, $\{g_t\}_{t\in(-1,1)}\subset\sG$ be an one parameter family with $g_0=g$ loxodromic and $X=dg_0(1)g^{-1}$. Then
		\[\tM(g,X)=\left.\frac{d}{dt}\right|_{t=0}\tJd(g_t).\]
	\end{proposition}
	\begin{proof}
		As $g_0=g$ is loxodromic, it follows that for $t$ small enough $g_t$ is loxodromic too. Hence, there exists $h_t\in\sG$ and $m_t\in\sM$ such that
		\[g_t=h_tm_t\exp(\tJd(g_t))h_t^{-1}=h_t\exp(\tJd(g_t))m_th_t^{-1}.\]
		Moreover, as $\sG$ is split we obtain that $\sM$ is discrete. Hence, $m_t=m\in\sM$ for small variations of $t$ and we obtain that
		\begin{align*}
			X=&\ \left.\frac{d}{dt}\right|_{t=0}g_tg^{-1}=\left.\frac{d}{dt}\right|_{t=0}h_t\exp(\tJd(g_t))mh_t^{-1}h_0m^{-1}\exp(-\tJd(g))h_0^{-1}\\
			=&\ dh_0(1)h_0^{-1}+\ \left.\frac{d}{dt}\right|_{t=0}h_0\exp(\tJd(g_t))\exp(-\tJd(g))h_0^{-1}\\
			&-h_0\exp(\tJd(g))mh_0^{-1}dh_0(1)h_0^{-1}h_0m^{-1}\exp(-\tJd(g))h_0^{-1}\\
			=&\ dh_0(1)h_0^{-1}+h_0\left(\left.\frac{d}{dt}\right|_{t=0}\tJd(g_t)\right)h_0^{-1}-gdh_0(1)h_0^{-1}g^{-1}.
		\end{align*}
		Therefore, we deduce that
		\[\tAd(h_0^{-1})[X]= \left.\frac{d}{dt}\right|_{t=0}\tJd(g_t)+ h_0^{-1}dh_0(1)-\tAd(h_0^{-1}gh_0)[h_0^{-1}dh_0(1)].\]
		Moreover, a straightforward computation shows that
		\[\pi_0(\tAd(h_0^{-1}gh_0)[h_0^{-1}dh_0(1)])=\pi_0[h_0^{-1}dh_0(1)].\]
		Finally, as $\tJd(g_t)\in\fa$, our result follows.
	\end{proof}
	
	Similar results relating Margulis invariants and Jordan projections has also been independently obtained by Andr\' es Sambarino \cite{Samba2} and Kassel--Smilga \cite{KaSm}.

	\subsection{Convexity of spectra}\label{sec.CS}
	Suppose $\rho\in\ano(\Gamma,\sG,\tR)$ and $(\rho,u):\Gamma\to\AG$ be an injective homomorphism. We use Proposition \ref{prop.partialaffineano} to obtain that for all infinite order elements $\gamma\in\Gamma$ the action of $({\rho,u})(\gamma)$ on the space of affine parabolic subspaces of $\sV$ has an attracting fixed point and a repelling fixed point. We abuse notation and let $\xi^+_{\rho,u}(\gamma_+)$ (resp. $\xi^-_{\rho,u}(\gamma_-)$) denote the attracting (resp. repelling) fixed point. Henceforth, we fix the representation $(\rho,u)$ and omit the subscripts $(\rho,u)$ from the notation of the Margulis invariants and the affine crossratios. Also, when there is no confusion of notation, for $a,b,c,d\in\bdry$ all distinct, we denote $\beta(\xi(a), \xi(b), \xi(c), \xi(d))$ by $\beta(a,b,c,d)$. In this subsection we show that the Margulis invariant spectra is a convex set.
	
	\begin{proposition}\label{prop.lim}
		Suppose $(\rho,u)$ is as above and $\gamma,\eta\in\Gamma$ are two infinite order elements such that the four points $\gamma_\pm,\eta_\pm\in\bdry$ are distinct. Then the following holds:
		\begin{enumerate}
			\item If the sequence $\{\gamma^m\eta^k\}_{m\in\mathbb{N}}\subset\Gamma$ contains a subsequence $\{\gamma^{n_i}\eta^k\}_{i\in\mathbb{N}}$ consisting only of infinite order elements. Then, 
			\begin{align*}
				\beta( \gamma_-,\eta_-,\gamma_+,\eta^k\gamma_+)&+
				\beta(\gamma_+,\eta_+,\gamma_-,\eta^{-k}\gamma_-)\\
				&= 	(1-\omega_0)\lim_{i\to\infty}[\tM(\gamma^{n_i}\eta^k)-          \tM(\gamma^{n_i})-\tM(\eta^k)].
			\end{align*}
			\item If the sequence $\{\gamma^m\eta^m\}_{m\in\mathbb{N}}\subset\Gamma$ contains a subsequence $\{\gamma^{n_i}\eta^{n_i}\}_{i\in\mathbb{N}}$ consisting only of infinite order elements. Then, 
			\[(1-\omega_0)\lim_{i\to\infty}(\tM(\gamma^{n_i}\eta^{n_i})-\tM(\gamma^{n_i})-\tM(\eta^{n_i}))=(1-\omega_0)\beta(\gamma_+,\eta_+,\gamma_-,\eta_-).\]
		\end{enumerate}
	\end{proposition}
	\begin{proof}
		We use Theorem \ref{thm.alphabeta} and follow verbatim the 	proof of Propositions 2.3.4 and 2.3.5 of Ghosh \cite{Ghosh3}.
	\end{proof}
	
	In fact, using an argument of Charette--Drumm \cite{cd} we obtain a finer version of Proposition \ref{prop.lim}. We start by making a few notational changes to make formulas look less tedious. 
	
	\begin{notation}
		Suppose $a,b\in\bdry$ and $\pi_{a,b}$ denote the projection on to $\sV^0_{a,b}:=(\sV_a\cap\sV_b)$ with respect to the splitting 
		\[\sV=\sV^+_a\oplus\sV^0_{a,b}\oplus\sV^-_b.\]
		We denote $\nu(a,b)$ by $\nu_{a,b}$ and $\nu^*(a,b)$ by $\nu^*_{a,b}$. Moreover, when $a=\gamma_+$ and $b=\gamma_-$ for some $\gamma\in\Gamma$, we denote $\sV^0_{a,b}$ by $\sV^0_\gamma$, $\pi_{a,b}$ by $\pi_\gamma$, $\nu_{a,b}$ by $\nu_\gamma$ and $\nu^*_{a,b}$ by $\nu^*_\gamma$. Observe that in the new notation for $n\in\N$ we have $\sV^0_{\gamma^n}=\sV^0_\gamma$, $\pi_{\gamma^n}=\pi_\gamma$, $\nu_{\gamma^n}=\nu_\gamma$ and $\nu^*_{\gamma^n}=\nu^*_\gamma$. 
	\end{notation}
	
	\begin{proposition}\label{prop.crossratio}
		Suppose $(\rho,u)$ is as above and $\gamma,\eta\in\Gamma$ are two infinite order elements such that the four points $\gamma_\pm,\eta_\pm\in\bdry$ are distinct and the sequence $\{\gamma^m\eta^m\}_{m\in\mathbb{N}}\subset\Gamma$ contains a subsequence $\{\gamma^{n_i}\eta^{n_i}\}_{i\in\mathbb{N}}$ consisting only of infinite order elements. Then the following identity holds:
		\[\lim_{i\to\infty}[\tM(\gamma^{n_i}\eta^{n_i})-\tM(\gamma^{n_i})-\tM(\eta^{n_i})]=\beta(\gamma_+,\eta_+,\gamma_-,\eta_-).\]
	\end{proposition}
	\begin{proof}
		Observe that $\sV^0_\gamma$ is the unit eigenspace of $\rho(\gamma)$ and suppose $\sA_\gamma$ is the unique affine subspace preserved under the action of  $(\rho(\gamma),u(\gamma))$. Let \[\sA^-_\gamma:=\sA_\gamma+\sV^-_{\gamma_-}.\] 
		As $\gamma$ and $\eta$ are coprime we have that $\sA_\eta$ intersects $\sA^-_\gamma$ in a unique point $Q$. Also let $R$ be the point on $\sA_\gamma$ such that 
		\[(R-Q)\in\sV^-_{\gamma_-}.\]
		We note that as $Q\in\sA_\eta$ the following holds
		\[Q - (\rho(\eta),u(\eta))^{-n}Q=\pi_\eta(Q - (\rho(\eta),u(\eta))^{-n}Q)=\pi_\eta(u(\eta^n)),\]
		and as $R\in\sA_\gamma$ the following holds
		\[(\rho(\gamma),u(\gamma))^nR - R =\pi_\gamma((\rho(\gamma),u(\gamma))^nR - R) = \pi_\gamma(u(\gamma^n)).\]
		Hence, we obtain that
		\begin{align*}
			\pi_{\gamma^n\eta^n}(u(\gamma^n\eta^n))=&\ \pi_{\gamma^n\eta^n}[(\rho(\gamma),u(\gamma))^nQ-(\rho(\gamma),u(\gamma))^nR-(Q-R)\\
			&+\ (Q - (\rho(\eta),u(\eta))^{-n}Q) +((\rho(\gamma),u(\gamma))^nR - R)] \\
			=&\ \pi_{\gamma^n\eta^n}[(\rho(\gamma)^n-\rho(e))(Q-R)+ \pi_\eta(u(\eta^n))+\pi_\gamma(u(\gamma^n))].
		\end{align*}
		Therefore, using Lemma \ref{lem.nu*} we get
		\[\tM(\gamma^n\eta^n)=\nu^*_{\gamma^n\eta^n}[(\rho(\gamma)^n-\rho(e))(Q-R)+ \pi_\eta(u(\eta^n))+\pi_\gamma(u(\gamma^n))].\]
		We use Lemma \ref{lem.nu*} to also note that
		\begin{align*}
			\tM(\gamma^n)&=\nu^*_\gamma(u(\gamma^n))=\nu^*_{\gamma_+,\eta_-}\pi_\gamma(u(\gamma^n)),\\
			\tM(\eta^n)&=\nu^*_\eta(u(\eta^n))=\nu^*_{\gamma_+,\eta_-}\pi_\eta(u(\eta^n)).
		\end{align*} 
		Moreover, as $\pi_\gamma(u(\gamma^n))=n\pi_\gamma(u(\gamma))$ and $\pi_\eta(u(\eta^n))=n\pi_\eta(u(\eta))$, the following holds:
		\begin{align*}
			\tM(\gamma^n\eta^n)-\tM(\gamma^n)-\tM(\eta^n)=&\ \nu^*_{\gamma^n\eta^n}(\rho(\gamma)^n-\rho(e))(Q-R)\\
			&+\ n[\nu^*_{\gamma^n\eta^n}-\nu^*_{\gamma_+,\eta_-}]\pi_\gamma(u(\gamma))\\
			&+\ n[\nu^*_{\gamma^n\eta^n}-\nu^*_{\gamma_+,\eta_-}]\pi_\eta(u(\eta)).
		\end{align*}
		As $\rho\in\ano(\Gamma,\sG,\tR)$, we note that $\nu^*_{\gamma^{n_i}\eta^{n_i}}$ converges exponentially to $\nu^*_{\gamma_+,\eta_-}$ and we obtain that 
		\[\lim_{i\to\infty}n_i[\nu^*_{\gamma^{n_i}\eta^{n_i}}-\nu^*_{\gamma_+,\eta_-}]=0.\]
		Furthermore, as $\rho(\gamma)$ is contracting on $\sV_{\gamma_-}^-$ and $(Q-R)\in\sV^-_{\gamma_-}$ we obtain
		\[\lim_{i\to\infty}\tM(\gamma^{n_i}\eta^{n_i})-\tM(\gamma^{n_i})-\tM(\eta^{n_i})=\nu^*_{\gamma_+,\eta_-}(R-Q).\]		
		Moreover, as $(R-Q)\in\sV^-_{\gamma_-}$, using Lemma \ref{lem.nu*} we obtain \[\nu^*_{\eta_+,\gamma_-}(R-Q)=\nu^*_{\eta_+,\gamma_-}\pi_{\eta_+,\gamma_-}(R-Q)=0.\]
		As $R\in\sA_\gamma$ and $Q\in\sA_\eta$, we observe that
		\[\beta(\gamma_+,\eta_+,\gamma_-,\eta_-)=\nu^*_{\gamma_+,\eta_-}(R-Q)+\nu^*_{\eta_+,\gamma_-}(R-Q)=\nu^*_{\gamma_+,\eta_-}(R-Q).\] 
		Our result follows.
	\end{proof}
	\begin{notation}
		Henceforth, we use the following notation:
		\begin{align*} \tM\text{-}\Spec(\rho,u)&:=\overline{\left\{\frac{\tM_{(\rho,u)}(\gamma)}{\ell(\gamma)}\mid \gamma\in\Gamma\right\}},\\
			\tJd\tM\text{-}\Spec(\rho,u)&:=\overline{\left\{\frac{(\tJd_\rho(\gamma),\tM_{(\rho,u)}(\gamma))}{\ell(\gamma)}\mid \gamma\in\Gamma\right\}}.
		\end{align*}
	\end{notation}
	\begin{proposition}\label{prop.convex}
		The set $\tM\text{-}\Spec(\rho,u)$ is convex.
	\end{proposition}
	\begin{proof}
		We use Proposition \ref{prop.crossratio} to deduce that 
		\[\lim_{n\to\infty}\frac{1}{n}\left[\tM_{(\rho,u)}(\gamma^n\eta^n)-\tM_{(\rho,u)}(\gamma^n)-\tM_{(\rho,u)}(\eta^n)\right]=0.\]
		Hence, it follows that 
		\begin{align*}
			\lim_{n\to\infty}\frac{\tM_{(\rho,u)}(\gamma^{pn}\eta^{qn})}{\ell(\gamma^{pn}\eta^{qn})}&=\lim_{n\to\infty}\frac{\tM_{(\rho,u)}(\gamma^{pn}\eta^{qn})/n}{\ell(\gamma^{pn}\eta^{qn})/n}=\frac{p \tM_{(\rho,u)}(\gamma)+q \tM_{(\rho,u)}(\eta)}{p \ell(\gamma)+q\ell(\eta)}\\
			&=\frac{p\ell(\gamma)}{p\ell(\gamma)+q\ell(\eta)}\frac{\tM_{(\rho,u)}(\gamma)}{\ell(\gamma)}+\frac{q\ell(\eta)}{p\ell(\gamma)+q\ell(\eta)}\frac{\tM_{(\rho,u)}(\eta)}{\ell(\eta)}.
		\end{align*}
		As the rationals are dense in the real numbers, our result follows.
	\end{proof}
	
	\begin{proposition}\label{prop.convexJM}
		Let $(\rho,u):\Gamma\to\AG$ be an injective homomorphism such that $[\rho]\in\ano(\Gamma,\sG,\tR)$. Then $\tJd\tM\text{-}\Spec(\rho,u)$ is a convex set.
	\end{proposition}
	
	\begin{proof}
		Follows from Corollary 4.1 of Benoist \cite{Ben1} and Proposition \ref{prop.convex}.
	\end{proof}
	
	We note that similar convexity results as that of Propositions \ref{prop.convex} and \ref{prop.convexJM} in a related, but slightly different, setting have also been obtained by Kassel--Smilga \cite{KaSm}.
	
	\section{Proper affine actions}\label{sec.4}
	
	\subsection{Criteria for proper actions}\label{sec.CPA}
	In this subsection we give various criterions for the existence of proper affine actions.
	
	\begin{definition}
		Suppose $\{\gamma_n\}_{n\in\N}\subset\Gamma$ is a sequence such that the word length of $\gamma_n$ goes to infinity as $n$ goes to infinity and
		\[\lim_{n\to\infty}\gamma_n^+=a\neq b=\lim_{n\to\infty}\gamma_n^-.\]
		Then we say that the sequence $\{\gamma_n\}_{n\in\N}$ is divergent.
	\end{definition}
	
	We observe that the translational length of a diverging sequence always diverges.
	
	\begin{theorem}\label{thm.proper}
		Let $\tR:\sG\to\SL(\sV)$ be an injective homomorphism, $\rho\in\ano(\Gamma,\sG,\tR)$ and $u$ be a $\sV$ valued cocycle with respect to $\rho$. Then the action of $(\rho,u)(\Gamma)$ on $\sV$ is not proper if and only if there exists a diverging sequence $\{\gamma_n\}_{n\in\N}$ inside $\Gamma$ such that ${\tM_{(\rho,u)}(\gamma_n)}$ stays bounded. 
	\end{theorem}
	\begin{proof}
		Suppose $\{\gamma_n\}_{n\in\N}$ be a diverging sequence of elements of $\Gamma$ such that $\tM_{(\rho,u)}(\gamma_n)$ stays bounded. Suppose $\sA_{\gamma_n}$ is the unique affine subspace which is preserved by $(\rho,u)(\gamma_n)$ and which is parallel to $\sV^0_{\gamma_n}$, the unit eigenspace of $(\rho,u)(\gamma_n)$. Suppose $\lim_{n\to\infty}\gamma_n^+=a$ and $\lim_{n\to\infty}\gamma_n^-=b$. We use Proposition \ref{prop.partialaffineano} to observe that $\sA_{\gamma_n}$ converge to $\sA_{a,b}$. We choose a sequence $\{X_n\}_{n\in\N}$, such that $X_n$ lies in  $\sA_{\gamma_n}$ and $\lim_{n\to\infty}X_n=X\in \sA_{a,b}$. Then 
		\[(\rho,u)(\gamma_n)X_n-X_n=\nu_{\gamma_n}\tM_{(\rho,u)}(\gamma_n)\in{\sV_{\gamma_n}^0}.\] 
		As $X_n$ converge to $X$ and $\tM_{(\rho,u)}(\gamma_n)$ stays bounded, we conclude that the sequence $\{(\rho,u)(\gamma_n)X_n\}_{n\in\N}$ also stays bounded. Therefore, the action of $(\rho,u)(\Gamma)$ on $\sV$ is not proper.
		
		On the other way around, suppose the action of $(\rho,u)(\Gamma)$ on $\sV$ is not proper. We use Theorem 5.9 of Guichard--Wienhard \cite{GW2} to deduce that there exists a diverging sequence $\{\gamma_n\}_{n\in\N}$ of elements of $\Gamma$ and another sequence $\{X_n\}_{n\in\N}$ of points in $\sV$ which converges to $X$ such that the sequence $\{(\rho,u)(\gamma_n)X_n\}_{n\in\N}$ also converges to some point $Y$ (for more details see the first two paragraphs of page 21 of Ghosh--Treib \cite{GT}). Suppose $\lim_{n\to\infty}\gamma_n^+=a$ and $\lim_{n\to\infty}\gamma_n^-=b$. Let $P_n$ (resp. $Q_n$) be that unique point on $\sA_{\gamma_n}$ such that 
		\[(X_n-P_n)\in{\sV^+_{\gamma_n^+}\oplus \sV^-_{\gamma_n^-}}\left(\text{resp.  \ }((\rho,u)(\gamma_n)X_n-Q_n)\in\sV^+_{\gamma_n^+}\oplus \sV^-_{\gamma_n^-}\right).\]
		We use Proposition \ref{prop.partialaffineano} to observe that $\sA_{\gamma_n}$ converge to $\sA_{a,b}$. Hence, we obtain that the sequence $\{P_n\}_{n\in\N}$ converge to some point $P\in\sA_{a,b}$ and $\{Q_n\}_{n\in\N}$ converge to some point $Q\in\sA_{a,b}$. Moreover, 
		\[\nu^*_{\gamma_n}(Q_n-P_n)=\tM_{(\rho,u)}(\gamma_n)\] 
		and it follows that $\tM_{(\rho,u)}(\gamma_n)$ stays bounded. 
	\end{proof}

	\begin{proposition}\label{prop.notproper}
		Let $\tR:\sG\to\SL(\sV)$ be an injective homomorphism, $\rho\in\ano(\Gamma,\sG,\tR)$ and $u$ be a $\sV$ valued cocycle with respect to $\rho$. Moreover, $(\rho,u)(\Gamma)$ does not act properly on $\sV$. Then $0\in\tM\text{-}\Spec(\rho,u)$.
	\end{proposition}
	\begin{proof}
		We use Proposition \ref{thm.proper} to obtain the existence of a sequence $\{\gamma_n\}_{n\in\N}$ inside $\Gamma$ such that as $n$ goes to infinity, $\ell(\gamma_n)$ diverges to infinity but $\tM_{(\rho,u)}(\gamma_n)$ stays bounded. It follows that 
		\[\lim_{n\to\infty}\frac{\tM_{(\rho,u)}(\gamma_n)}{\ell(\gamma_n)}=0.\]
		Hence, $0\in\tM\text{-}\Spec(\rho,u)$.
	\end{proof}
	
	We note that similar results as Theorem \ref{thm.proper} and Proposition \ref{prop.notproper} in a related, but slightly different, setting have independently been obtained by Kassel--Smilga \cite{KaSm}.
	
	\begin{lemma}\label{lem.margnu}
		Suppose $t_\gamma$ is the period of the periodic orbit corresponding to $\gamma\in\Gamma$, $\sigma:\flow\to\acV$ is a neutralized section and  $\tilde{\sigma}:\cflow\to\sV$ be the map induced by the lift of $\sigma$. Then the following holds:
		\[\nu_\gamma\tM_{\rho,u}(\gamma)=\int_0^{t_\gamma}\nabla_\phi\tilde{\sigma}(\gamma_+,\gamma_-,s)ds .\]
	\end{lemma}
	\begin{proof}
		We observe that
		\begin{align*}                  \int_0^{t_\gamma}\nabla_\phi\tilde{\sigma}(\gamma_+,\gamma_-,s)ds&=\pi_{\gamma}\left[\int_0^{t_\gamma}\nabla_\phi\tilde{\sigma}(\gamma_+,\gamma_-,s)ds\right]\\
			&=\pi_{\gamma}[\tilde{\sigma}(\gamma_+,\gamma_-,t_\gamma)-\tilde{\sigma}(\gamma_+,\gamma_-,0)]\\
			&=\pi_{\gamma}[\rho(\gamma)\tilde{\sigma}(\gamma_+,\gamma_-,0)+u(\gamma)-\tilde{\sigma}(\gamma_+,\gamma_-,0)]\\
			&=\pi_{\gamma}[u(\gamma)]=\nu_\gamma\nu^*_\gamma[ u(\gamma)]=\nu_\gamma\tM_{\rho,u}(\gamma).
		\end{align*}
		We conclude our result.   
	\end{proof}
	
	\begin{proposition}\label{prop.proper}
		Let $\tR:\sG\to\SL(\sV)$ be an injective homomorphism, $\rho\in\ano(\Gamma,\sG,\tR)$ and $u$ be a $\sV$ valued cocycle with respect to $\rho$. Moreover, let $\sL$ be an one dimensional vector subspace of $\sV^0$ and $0\in \tM\text{-}\Spec(\rho,u)\subset \sL$. Then the action of $(\rho,u)(\Gamma)$ on $\sV$ is not proper.
	\end{proposition}
	\begin{proof}
		Suppose $\sigma:\flow\to\cV$ be a H\" older continuous section which is differentiable along flow lines of $\flow$ and $\sigma_0:\flow\to\sV^0$ be such that $\nabla_\phi^0\sigma=\nabla_\phi\sigma_0$. Now without loss of generality we choose $0\neq E_1\in\sL$ and extend it to a basis $\{E_k\}_{k=1}^m$ of $\sV^0$. Then there exists functions $f_i:\flow\to\R$ for all $1\leq i\leq m$ such that $\nabla_\phi\sigma_0=f_1\nu_{E_1}+\dots+f_m\nu_{E_m}$. As $\tM\text{-}\Spec(\rho,u)\subset \sL$, we use Lemma \ref{lem.margnu} to observe that for all non torsion element $\gamma\in\Gamma$,
		\[\left(\int_0^{t_\gamma}f_1(\gamma_+,\gamma_-,s)ds\right)E_1=\tM_{\rho,u}(\gamma),\]
		and for all $2\leq i\leq m$ we have
		\[\int_0^{t_\gamma}f_i(\gamma_+,\gamma_-,s)ds=0.\]
		We use Liv\v sic's theorem \cite{Liv} to deduce that for all $2\leq i\leq m$ there exists $g_i:\flow\to\R$ such that $f_i=\frac{\partial g_i}{\partial t}$. Hence, the section $\sigma_1:=[\sigma_0-(g_2\nu_{E_2}+\dots+g_m\nu_{E_m})]$ satisfies
		\[\nabla_\phi\sigma_1=f_1\nu_{E_1}.\]
		Moreover, as $0\in\tM\text{-}\Spec(\rho,u)$ and the metric $d$ on $\cflow$ is bi-Lipschitz to the product metric (see Corollary 8.3H of Gromov \cite{Gromov}), there exists a diverging sequence $\{\gamma_n\}_{n\in\N}$ such that
		\[\lim_{n\to\infty}\frac{\tM_{\rho,u}(\gamma_n)}{t_{\gamma_n}}=0.\]
		Suppose $\mu_{\gamma_n}$ is a $\phi$ invariant probability measure on $\flow$ which is supported on the orbit of the flow corresponding to $\gamma_n$. Then
		\[\lim_{n\to\infty}\int f_1d\mu_{\gamma_n}=0.\]
		We know that the space of $\phi$ invariant probability measures on $\flow$ is weak* compact. Hence, there exists a $\phi$ invariant probability measure $\mu$ on $\flow$ such that
		\[\int f_1d\mu=0.\]
		We consider $f^t_1(p):=\frac{1}{t}\int_0^tf(\phi_sp)ds$ for all $t>0$ and using Fubini's theorem obtain that
		\[\int f^t_1d\mu=0.\]
		As $\flow$ is connected, for all $t>0$, there exists $p_t\in\flow$ such that $f^t_1(p_t)=0$. Hence, for all $t>0$ we have $\sigma_1(\phi_tp_t)=\phi_t\sigma_1(p_t)$. It follows that for the compact set $\sigma_1(\flow)$ and $t>0$ the following holds
		\[\sigma_1(\phi_tp_t)\in\phi_t\sigma_1(\flow)\cap\sigma_1(\flow).\]
		We deduce that the $\R$ action on $\acV$ is not proper. Now using Lemma 5.2 of Goldman--Labourie--Margulis \cite{GLM} we obtain that the $\Gamma$ action on $(\cflow\times\sV)/\R=\bdry^{(2)}\times\sV$ is not proper. Hence, the action of $(\rho,u)(\Gamma)$ on $\sV$ is not proper either.
	\end{proof}
	\subsection{Deformation of free groups}\label{sec.DFG}
	
	Henceforth, we denote $\SL(n,\R)$ by $\SL_n$ for notational ease. In this subsection we assume that $\Gamma$ is a non-abelian free group. Let $\iota:\SL_2\to\SL_n$ be the irreducible representation and let $d_e\iota:\fsl_2\to\fsl_n$ be the corresponding Lie algebra homomorphism. Then any injective representation $\rho:\Gamma\to\SL_n$ is called a Fuchsian representation if and only if $\rho=\iota\circ\tau$ for some $\tau:\Gamma\to\SL_2$. Moreover, any representation $\rho:\Gamma\to\SL_n$ which can be continuously deformed to obtain a Fuchsian representation is called a Hitchin representation. We denote the space of Hitchin representations by $\hit(\Gamma,\SL_n)$. Suppose $\rho$ is a Fuchsian representation and $\sT_{[\rho]}\hit(\Gamma,\SL_n)$ is the tangent space at $[\rho]$.
	
	\begin{lemma}\label{lem.existence}
		Let $(\tau,u):\Gamma\to\SL_2\ltimes\fsl_2$ be an injective homomorphism such that $\tau(\Gamma)$ is a non-abelian free subgroup of $\SL_2$ and $(\tau,u)(\Gamma)$ acts properly on $\fsl_2$. Then the group $(\iota\circ\tau,d_e\iota \circ u)(\Gamma)$ acts properly on $\fsl_n$.
	\end{lemma}
	\begin{proof}
		We observe that for $k=\left\lfloor{\frac{n}{2}}\right\rfloor $ we have
		\[\iota\left(\begin{bmatrix}
			\lambda&0\\
			0&\lambda^{-1}
		\end{bmatrix}\right)=
		\begin{bmatrix}
			\lambda^k&\dots&0\\
			\vdots&\ddots&\vdots\\
			0&\dots&\lambda^{-k}
		\end{bmatrix}.\]
		Suppose $\{\tau_t(\Gamma)\}_{t\in(-1,1)}$ be an one parameter family of convex cocompact subgroups of $\SL_2$ with $\tau_0=\tau$ and whose tangent direction at $\tau$ is the cocycle $u:\Gamma\to\fsl_2$. Then the cocycle $d_e\iota\circ u:\Gamma\to\fsl_n$ is tangent to the one parameter family $\{\iota\circ\tau_t\}_{t\in(-1,1)}$. We use Goldman--Margulis \cite{GM} and Goldman--Labourie--Margulis \cite{GLM} to obtain that for some $c>0$ the following holds:
		\[\left.\frac{d}{dt}\right|_{t=0}\frac{\tJd(\tau_t(\gamma))}{\ell(\gamma)}\geq c>0.\]
		Hence, we can choose a linear functional $\alpha:\fsl_n^0\to\R$ such that
		\[\left.\frac{d}{dt}\right|_{t=0}\alpha\left(\frac{\tJd(\iota\circ\tau_t(\gamma))}{\ell(\gamma)}\right)\geq c>0.\]
		Now using Proposition \ref{prop.derivative} it follows that $0\notin\tM\text{-}\Spec(\iota\circ\tau,d_e\iota\circ u)$. Finally, using Proposition \ref{prop.notproper} we conclude our result.
	\end{proof}
	
	\begin{proposition}\label{prop.existence}
		Let $(\rho,u):\Gamma\to\SL_n\ltimes\fsl_n$ be an injective homomorphism such that $0\notin\tM\text{-}\Spec(\rho,u)$. Then there exists some open neighborhood $U\subset\sT\hit(\Gamma,\SL_n)$ containing $([\rho],[u])$ such that for all $([\varrho],[v])\in U$, the group $(\varrho,v)(\Gamma)$ act properly on $\fsl_n$.
	\end{proposition}
	\begin{proof}
		As $0\notin\tM\text{-}\Spec(\rho,u)$, there exists some linear functional $\alpha:\fsl_n\to\R$ and $c>0$ such that $\alpha(\tM_{\rho,u}(\gamma))\geq c\ell(\gamma)>0$ for all $\gamma\in\Gamma$.
		
		Let $\cA_{\rho,u}:=(\rho,u)(\Gamma)\backslash(\cflow\times\fsl_n)$. We use Theorem 3.8 of Hirsch-Pugh-Shub \cite{HPS} (see Section 6 of Bridgeman--Canary--Labourie--Sambarino \cite{BCLS} for more details) to obtain that there exists a family of H\" older continuous sections 
		\[\{\sigma_{\rho,u}:\flow\to \cA_{\rho,u}\mid [\rho]\in\hit(\Gamma,\SL_n),[u]\in\sT_{[\rho]}\hit(\Gamma,\SL_n)\}\] 
		which are differentiable along flow lines such that $\sigma_{\rho,u}$ varies analytically over $\sT\hit(\Gamma,\SL_n)$. We use Lemma \ref{lem.neutralized1} to observe that we could choose $\sigma_{\rho,u}$ such that
		\[\nabla_\phi\sigma_{\rho,u}=\nabla_\phi^0\sigma_{\rho,u}.\]
		Moreover, we use Lemma \ref{lem.margnu} to deduce the existence of H\" older continuous functions $f_{\rho,u}:\flow\to\fsl_n^0$ which varies analytically over $\sT\hit(\Gamma,\SL_n)$ such that for all non torsion elements $\gamma\in\Gamma$ the following holds:
		\[\int_\gamma f_{\rho,u}=\frac{\tM_{\rho,u}(\gamma)}{t_\gamma}.\]
		We use Corollary 8.3H of Gromov \cite{Gromov} and obtain some $c_0>0$ such that for all non torsion elements $\gamma\in\Gamma$ we have 
		\[\int_\gamma \alpha(f_{\rho,u})\geq c_0 >0.\]
		We use Lemmas 3 and 6 of Goldman--Labourie \cite{GL} to obtain that we could choose $\alpha(f_{\rho,u})>0$. As $\flow$ is compact, there exists some $c_1>0$ such that $\alpha(f_{\rho,u})\geq c_1$. Also, as the dependence over $\sT\hit(\Gamma,\SL_n)$ is analytic, we deduce that for any $(\varrho,v)$ in a small neighborhood containing $(\rho,u)$ we have $\alpha(f_{\varrho,v})>0$. It follows that $0\notin\tM\text{-}\Spec(\varrho,v)$. Finally, using Proposition \ref{prop.notproper} we conclude our result.
	\end{proof}
	
	\begin{proposition}\label{prop.main}
		Suppose $\rho:\Gamma\to\SL_n$ is a Fuchsian representation. Then there exists a neighborhood $U$ of $[\rho]$ in $\hit(\Gamma,\SL_n)$ and for any $[\varrho]\in U$ there exists some non empty open set $U_{[\varrho]}\subset\sT_{[\varrho]}\hit(\Gamma,\SL_n)$ such that for any $v\in U_{[\varrho]}$, the group $(\varrho,v)(\Gamma)$ act properly on $\fsl_n$.
	\end{proposition}
	\begin{proof}
		Suppose $\rho=\iota\circ\tau$ for some free subgroup $\tau(\Gamma)$. We know from Margulis \cite{Margulis1,Margulis2} that there exists some cocycle $u:\Gamma\to\fsl_2$ such that $(\tau,u)(\Gamma)$ acts properly on $\fsl_2$. Now we use Lemma \ref{lem.existence} and Proposition \ref{prop.existence} to conclude our result.
	\end{proof}

	\subsection{Deformation of surface groups}\label{sec.DSG}
	
	In this subsection  we assume that $\Gamma$ is the fundamental group of a compact surface without boundary of genus atleast two. Let $\iota:\SL_2\to\SL_n$ be the irreducible representation. Then any injective representation $\rho:\Gamma\to\SL_n$ is called a Fuchsian representation if and only if $\rho=\iota\circ\tau$ for some $\tau:\Gamma\to\SL_2$. Moreover, any representation $\rho:\Gamma\to\SL_n$ which can be continuously deformed to obtain a Fuchsian representation is called a Hitchin representation. We denote the space of Hitchin representations by $\hit(\Gamma,\SL_n)$. Suppose $\rho$ be a Fuchsian representation and $\sT_{[\rho]}\hit(\Gamma,\SL_n)$ be the tangent space at $[\rho]$. 
	\begin{proposition}\label{prop.onedim}
		Suppose $\rho:\Gamma\to\SL_n$ is a Fuchsian representation. Then there exists subspaces $\sT_2,\dots,\sT_n$ of $\sT_{[\rho]}\hit(\Gamma,\SL_n)$ such that
		\[\sT_{[\rho]}\hit(\Gamma,\SL_n)=\bigoplus_{k=2}^n\sT_k\]
		and for any $[u_k]\in\sT_k$ there exists one dimensional subspaces $\sL_k$ of $\fsl_n^0$ spanned respectively by $X_k=(X_{k,2},\dots,X_{k,n})$ with
		\[X_{k,p}=\frac{(p-1)!(n-p)!}{2^{k-2}(n-k)!}\sum_{j=\max\{1,k+p-n\}}^{\min\{k,p\}} {n-k\choose p-j}{k-1\choose j-1}^2(-1)^{j+k+1},\]
		such that $\tM\text{-}\Spec(\rho,u_k)\subset \sL_k$.
	\end{proposition}
	\begin{proof}
		We use Proposition \ref{prop.derivative} and Theorem 4.0.2 of Labourie--Wentworth \cite{LW} to obtain our result.
	\end{proof}
	\begin{remark}
		The subspaces $\sT_k$ correspond to the spaces of $k$-differentials (see Labourie--Wentworth \cite{LW} for more details).
	\end{remark}
	
	\begin{notation}
		Let $\left(\fa^{+}\right)^*$ denote the collection of all linear functionals $\alpha:\fsl_n^0\to\R$ such that $\alpha(X)\geq0$ for all $X\in\fa^{+}$.
	\end{notation}
	
	\begin{proposition}\label{prop.zerospec}
		Let $(\rho,u):\Gamma\to\SL_n\ltimes\fsl_n$ be an injective homomorphism such that $\rho$ is a Fuchsian representation and $[u]\in\sT_{[\rho]}\hit(\Gamma,\SL_n)$. Then 
		$0\in\alpha\left(\tM\text{-}\Spec(\rho,u)\right)$ for any $\alpha\in\left(\fa^{+}\right)^*$.
	\end{proposition}
	\begin{proof}
		If possible we assume that $0\notin\alpha\left(\tM\text{-}\Spec(\rho,u)\right)$. We use Proposition \ref{prop.convex} to deduce that either $\alpha\left(\tM\text{-}\Spec(\rho,u)\right)>0$ or $\alpha\left(\tM\text{-}\Spec(\rho,u)\right)<0$. Hence, we observe that up to a choice between $\pm u$, there exists some $c>0$ such that for all non torsion $\gamma\in\Gamma$ we have 
		\[\alpha(\tM_{(\rho,u)}(\gamma))\leq -c\ell(\gamma).\] 
		Suppose $\{\rho_s\mid\rho_s:\Gamma\to\SL_n, s\in(-1,1)\}$ is an analytic one parameter family of Hitchin representations whose tangent at $\rho=\rho_0$ is given by the cocycle $u:\Gamma\to\fsl_n$. We use results from Section 6.3 of Bridgeman--Canary--Labourie--Sambarino \cite{BCLS} to obtain an analytic family of functions
		$f_s:\flow\to\fsl_n^0$ such that
		\[\int_\gamma f_s=\frac{\tJd_{\rho_t}(\gamma)}{t_\gamma}.\]
		Hence, $\alpha(f_s)$ also varies analytically. It follows that for some $g_s:\flow\to\fsl_n^0$,
		\[\alpha(f_s)=\alpha(f_0)+s\alpha(f^\prime_0)+\frac{s^2}{2!}\alpha(g_s).\]
		Now using Proposition \ref{prop.derivative} and Corollary 8.3H of Gromov \cite{Gromov} we obtain some $c_0>0$ such that for all $\gamma\in\Gamma$ the following holds:
		\[\int_\gamma\alpha(f^\prime_0)=\alpha\left(\frac{\tM_{\rho,u}(\gamma)}{t_\gamma}\right)\leq -c_0<0.\]
		We use Lemmas 3 and 6 of Goldman--Labourie \cite{GL} to observe that we can choose $f_s$ such that $\alpha(f^\prime_0)<0$. As $\flow$ is compact, there exists some $c_1>0$ and $b_s>0$ such that $\alpha(f^\prime_0)\leq -c_1$ and $\alpha(g_s)\leq b_s$. Therefore, for small enough $s>0$, there exists $c_2>0$ such that
		\[\alpha(f_s)\leq\alpha(f_0)-sc_2\leq\alpha(f_0).\]
		Hence, for all $\gamma\in\Gamma$ we have $\alpha(\tJd_{\rho_s}(\gamma))\leq\alpha(\tJd_{\rho}(\gamma))$ and we deduce that
		\[\lim_{t\to\infty}\frac{1}{t}\log|\{\gamma\mid\alpha(\tJd_{\rho}(\gamma))\leq t\}|\leq\lim_{t\to\infty}\frac{1}{t}\log|\{\gamma\mid\alpha(\tJd_{\rho_s}(\gamma))\leq t\}|.\]
		Now using Corollary 1.4 of Potrie--Sambarino \cite{PotSam} we deduce that $\rho_s$ is a Fuchsian representation for all small enough $s>0$. Finally, we use results from Goldman--Margulis \cite{GM} to obtain that $0\in\tM\text{-}\Spec(\rho,u)$. Hence,  $0\in\alpha\left(\tM\text{-}\Spec(\rho,u)\right)$, a contradiction to our initial assumption. Therefore, our result follows.    
	\end{proof}

	\begin{proposition}\label{prop.hitzerospec}
		Suppose $(\rho,u_k):\Gamma\to\SL_n\ltimes\fsl_n$ is an injective homomorphism such that $\rho$ is a Fuchsian representation and $[u_k]\in\sT_k$. Then $0\in\tM\text{-}\Spec(\rho,u_k)$.
	\end{proposition}
	\begin{proof}
		We recall from Proposition \ref{prop.onedim} that $\tM\text{-}\Spec(\rho,u_k)\subset \sL_k$ and $\sL_k$ is a one dimensional vector subspace. As $\left(\fa^+\right)^*$ has non-empty interior, we can choose $\alpha\in\left(\fa^+\right)^*$ such that the kernel of $\alpha$ is transverse to $\sL_k$. It follows that $0\in\tM\text{-}\Spec(\rho,u_k)$ if and only if $0\in\alpha\left(\tM\text{-}\Spec(\rho,u_k)\right)$. Finally, using Proposition \ref{prop.zerospec} we conclude our result.
	\end{proof}
	
	\begin{proposition}
		Suppose $(\rho,u):\Gamma\to\SL_n\ltimes\fsl_n$ is an injective homomorphism such that $\rho$ is a Fuchsian representation and $[u_k]\in\sT_k$. Then $(\rho,u_k)(\Gamma)$ does not act properly on $\fsl_n$.
	\end{proposition}
	\begin{proof}
		The result follows from Propositions \ref{prop.proper} and \ref{prop.hitzerospec}.
	\end{proof}
	
	\begin{remark}
		The above proposition can also be proved more directly by considering the complete reducibility of $\fsl_2$ representations and Theorem 1.1 of Labourie \cite{Labourie2}. We thank the anonymous referee for bringing this to our notice.
	\end{remark}
	
	\begin{notation}
		We use the following notation:
		\[\sT_{odd}:=\bigoplus_{m=1}^{\lfloor \frac{n-1}{2}\rfloor}\sT_{2m+1} \text{ and } \sT_{even}:=\bigoplus_{m=1}^{\lfloor \frac{n}{2}\rfloor}\sT_{2m}.\]
	\end{notation}
	
	\begin{proposition}
		Suppose $(\rho,u):\Gamma\to\SL_n\ltimes\fsl_n$ is an injective homomorphism such that $\rho$ is a Fuchsian representation and $u\in\sT_{odd}$. Then $(\rho,u)(\Gamma)$ does not act properly on $\fsl_n$.
	\end{proposition}
	\begin{proof}
		As $u$ corresponds to an odd differential, using Equation (5), Propositions 2.1.1 and 4.0.2 of Labourie--Wentworth \cite{LW} we obtain that 
		\[\tM_{\rho,u}\left(\gamma^{-1}\right)=-\tM_{\rho,u}(\gamma)\] for all $\gamma\in\Gamma$. Hence, for any $\gamma,\eta\in\Gamma$ we have
		\[\tM_{\rho,u}\left(\eta\gamma^n\eta^{-1}\right)+\tM_{\rho,u}\left(\gamma^{-n}\right)=0.\]            
		We choose the elements $\gamma$ and $\eta$ from $\Gamma$ such that they are of infinite order and transverse. It follows that, we have $\{\eta\gamma^n\eta^{-1}\gamma^{-n}\}_{n=0}^\infty$ is a diverging sequence. Moreover, using Proposition \ref{prop.crossratio} we observe that 
		\begin{align*}
			\beta(\eta\gamma_+,\gamma_-,\eta\gamma_-,\gamma_+)&= \lim_{n\to\infty}[\tM_{\rho,u}\left(\eta\gamma^n\eta^{-1}\gamma^{-n}\right)-\tM_{\rho,u}\left(\eta\gamma^n\eta^{-1}\right)-\tM_{\rho,u}\left(\gamma^{-n}\right)]\\ &=\lim_{n\to\infty}\tM_{\rho,u}\left(\eta\gamma^n\eta^{-1}\gamma^{-n}\right).
		\end{align*}
		Finally, we conclude our result using Theorem \ref{thm.proper}.
	\end{proof}
	\begin{theorem}\label{thm.main}
		Suppose $(\rho,u):\Gamma\to\SL_n\ltimes\fsl_n$ is an injective homomorphism such that $\rho$ is a Fuchsian representation and $u\in\sT_{odd}\oplus\sT_{2m}$ for some $1\leq m\leq \lfloor \frac{n}{2}\rfloor$. Then $(\rho,u)(\Gamma)$ does not act properly on $\fsl_n$.
	\end{theorem}
	\begin{proof}
		Suppose $u=u_{odd}+u_{2m}$ with $u_{odd}\in\sT_{odd}$ and $u_{2m}\in\sT_{2m}$. 
		
		On one hand, as $(\rho,u_{2m})(\Gamma)$ does not act properly on $\fsl_n$, using Theorem \ref{thm.proper} we obtain a diverging sequence $\{\gamma_n\}_{n=0}^\infty$ such that $\{\tM_{\rho,u_{2m}}(\gamma_n)\}_{n=0}^\infty$ stays bounded. Suppose the attracting fixed points of $\gamma_n$ on $\bdry$ converge to $a\in\bdry$ and the repelling fixed points of $\gamma_n$ on $\bdry$ converge to $b\in\bdry$. Clearly, $a\neq b$. We choose an element $\eta\in\Gamma$ whose axis is transverse to the axis between $a$ and $b$. Hence, $\{\eta\gamma_n\eta^{-1}\gamma_n^{-1}\}_{n=0}^\infty$ is a diverging sequence and by Proposition 9.3 of Smilga \cite{Smilga4} we have
		\[\lim_{n\to\infty} \|\tM_{\rho,u_{2m}}\left(\eta\gamma_n\eta^{-1}\gamma_n^{-1}\right)-\tM_{\rho,u_{2m}}\left(\eta\gamma_n\eta^{-1}\right)-\tM_{\rho,u_{2m}}\left(\gamma_n^{-1}\right)\|<\infty.\]
		We observe that for all $\gamma\in\Gamma$,
		\[\tM_{\rho,u_{2m}}(\gamma^{-1})=\tM_{\rho,u_{2m}}(\gamma)=\tM_{\rho,u_{2m}}(\eta\gamma\eta^{-1}),\]  
		and the sequence $\{\tM_{\rho,u_{2m}}(\gamma_n)\}_{n=0}^\infty$ stays bounded. Therefore, we deduce that $\{\tM_{\rho,u_{2m}}\left(\eta\gamma_n\eta^{-1}\gamma_n^{-1}\right)\}_{n=0}^\infty$ also stays bounded.
		
		On the other hand, as $\{\eta\gamma_n\eta^{-1}\gamma_n^{-1}\}_{n=0}^\infty$ is a diverging sequence, we again use Proposition 9.3 of Smilga \cite{Smilga4} to obtain that
		\[\lim_{n\to\infty} \|\tM_{\rho,u_{odd}}\left(\eta\gamma_n\eta^{-1}\gamma_n^{-1}\right)-\tM_{\rho,u_{odd}}\left(\eta\gamma_n\eta^{-1}\right)-\tM_{\rho,u_{odd}}\left(\gamma_n^{-1}\right)\|<\infty.\]
		This time, we observe that for all $\gamma\in\Gamma$,
		\[\tM_{\rho,u_{odd}}(\eta\gamma\eta^{-1}) + \tM_{\rho,u_{odd}}(\gamma^{-1})=0.\] 
		Therefore, it follows that
		$\{\tM_{\rho,u_{odd}}\left(\eta\gamma_n\eta^{-1}\gamma_n^{-1}\right)\}_{n=0}^\infty$ stays bounded too.
		
		Finally, as $\tM_{\rho,u}(\gamma)=\tM_{\rho,u_{odd}}(\gamma)+\tM_{\rho,u_{2m}}(\gamma)$ for all $\gamma\in\Gamma$, we conclude that the sequence $\{\tM_{\rho,u}\left(\eta\gamma_n\eta^{-1}\gamma_n^{-1}\right)\}_{n=0}^\infty$ stays bounded and by Theorem \ref{thm.proper} our result follows.
	\end{proof}

	\bibliography{Library.bib}

\begin{thebibliography}{GGKW17}

\bibitem[AMS02]{AMS}
H.~Abels, G.~A. Margulis, and G.~A. Soifer.
\newblock On the {Z}ariski closure of the linear part of a properly
  discontinuous group of affine transformations.
\newblock {\em J. Differential Geom.}, 60(2):315--344, 2002.

\bibitem[AMS20]{AMS2}
H.~{Abels}, G.~A. {Margulis}, and G.~A. {Soifer}.
\newblock {The Auslander conjecture for dimension less then 7}.
\newblock {\em arXiv e-prints}, page arXiv:2011.12788, November 2020.

\bibitem[Aus64]{Aus}
L.~Auslander.
\newblock The structure of complete locally affine manifolds.
\newblock {\em Topology}, 3(suppl. 1):131--139, 1964.

\bibitem[BCLS15]{BCLS}
M.~Bridgeman, R.~Canary, F.~Labourie, and A.~Sambarino.
\newblock The pressure metric for {A}nosov representations.
\newblock {\em Geom. Funct. Anal.}, 25(4):1089--1179, 2015.

\bibitem[Ben97]{Ben1}
Y.~Benoist.
\newblock Propri\'{e}t\'{e}s asymptotiques des groupes lin\'{e}aires.
\newblock {\em Geom. Funct. Anal.}, 7(1):1--47, 1997.

\bibitem[Ben00]{Ben2}
Yves Benoist.
\newblock Propri\'{e}t\'{e}s asymptotiques des groupes lin\'{e}aires. {II}.
\newblock In {\em Analysis on homogeneous spaces and representation theory of
  {L}ie groups, {O}kayama--{K}yoto (1997)}, volume~26 of {\em Adv. Stud. Pure
  Math.}, pages 33--48. Math. Soc. Japan, Tokyo, 2000.

\bibitem[Bie11]{B1}
Ludwig Bieberbach.
\newblock \"{U}ber die {B}ewegungsgruppen der {E}uklidischen {R}\"{a}ume.
\newblock {\em Math. Ann.}, 70(3):297--336, 1911.

\bibitem[Bie12]{B2}
Ludwig Bieberbach.
\newblock \"{U}ber die {B}ewegungsgruppen der {E}uklidischen {R}\"{a}ume
  ({Z}weite {A}bhandlung.) {D}ie {G}ruppen mit einem endlichen
  {F}undamentalbereich.
\newblock {\em Math. Ann.}, 72(3):400--412, 1912.

\bibitem[BPS19]{BPS}
Jairo Bochi, Rafael Potrie, and Andr\'{e}s Sambarino.
\newblock Anosov representations and dominated splittings.
\newblock {\em J. Eur. Math. Soc. (JEMS)}, 21(11):3343--3414, 2019.

\bibitem[BQ16]{BQ}
Yves Benoist and Jean-Fran\c{c}ois Quint.
\newblock {\em Random walks on reductive groups}, volume~62 of {\em Ergebnisse
  der Mathematik und ihrer Grenzgebiete. 3. Folge. A Series of Modern Surveys
  in Mathematics [Results in Mathematics and Related Areas. 3rd Series. A
  Series of Modern Surveys in Mathematics]}.
\newblock Springer, Cham, 2016.

\bibitem[Can21]{Can}
Richard Canary.
\newblock Anosov representations: Informal lecture notes.
\newblock October 10, 2021.

\bibitem[CD04]{cd}
V.~Charette and T.~Drumm.
\newblock Strong marked isospectrality of affine {L}orentzian groups.
\newblock {\em J. Differential Geom.}, 66(3):437--452, 03 2004.

\bibitem[Cha94]{Champetier}
C.~Champetier.
\newblock Petite simplification dans les groupes hyperboliques.
\newblock {\em Annales de la Facult\'e des sciences de Toulouse :
  Math\'ematiques}, 3(2):161--221, 1994.

\bibitem[DDGS22]{DDGS}
Jeffrey Danciger, Todd~A. Drumm, William~M. Goldman, and Ilia Smilga.
\newblock Proper actions of discrete groups of affine transformations.
\newblock In {\em Dynamics, geometry, number theory---the impact of {M}argulis
  on modern mathematics}, pages 95--168. Univ. Chicago Press, Chicago, IL,
  [2022] \copyright 2022.

\bibitem[DGK16]{DGK2}
J.~Danciger, F.~Gu{\'e}ritaud, and F.~Kassel.
\newblock Margulis spacetimes via the arc complex.
\newblock {\em Invent. Math.}, 204(1):133--193, 2016.

\bibitem[DGK20]{DGK3}
Jeffrey Danciger, Fran\c{c}ois Gu\'{e}ritaud, and Fanny Kassel.
\newblock Proper affine actions for right-angled {C}oxeter groups.
\newblock {\em Duke Math. J.}, 169(12):2231--2280, 2020.

\bibitem[Dru93]{Drumm2}
T.~A. Drumm.
\newblock Linear holonomy of {M}argulis space-times.
\newblock {\em J. Differential Geom.}, 38(3):679--690, 1993.

\bibitem[DZ19]{DZ}
Jeffrey Danciger and Tengren Zhang.
\newblock Affine actions with {H}itchin linear part.
\newblock {\em Geom. Funct. Anal.}, 29(5):1369--1439, 2019.

\bibitem[FG83]{FG}
D.~Fried and W.~M. Goldman.
\newblock Three-dimensional affine crystallographic groups.
\newblock {\em Adv. in Math.}, 47(1):1--49, 1983.

\bibitem[GGKW17]{GGKW}
Fran\c{c}ois Gu\'{e}ritaud, Olivier Guichard, Fanny Kassel, and Anna Wienhard.
\newblock Anosov representations and proper actions.
\newblock {\em Geom. Topol.}, 21(1):485--584, 2017.

\bibitem[Gho17]{Ghosh2}
S.~Ghosh.
\newblock Anosov structures on {M}argulis spacetimes.
\newblock {\em Groups Geom. Dyn.}, 11(2):739--775, 2017.

\bibitem[{Gho}18]{Ghosh3}
Sourav {Ghosh}.
\newblock {Avatars of Margulis invariants and proper actions}.
\newblock {\em arXiv e-prints}, page arXiv:1812.03777, Dec 2018.

\bibitem[GL12]{GL}
W.~M. Goldman and F.~Labourie.
\newblock Geodesics in {M}argulis spacetimes.
\newblock {\em Ergodic Theory Dynam. Systems}, 32(2):643--651, 2012.

\bibitem[GLM09]{GLM}
W.~M. Goldman, F.~Labourie, and G.~Margulis.
\newblock Proper affine actions and geodesic flows of hyperbolic surfaces.
\newblock {\em Ann. of Math. (2)}, 170(3):1051--1083, 2009.

\bibitem[GM00]{GM}
William~M. Goldman and Gregory~A. Margulis.
\newblock Flat {L}orentz 3-manifolds and cocompact {F}uchsian groups.
\newblock In {\em Crystallographic Groups and Their Generalizations: Workshop,
  Katholieke Universiteit Leuven Campus Kortrijk, Belgium, May 26-28, 1999},
  Contemporary mathematics - American Mathematical Society, pages 135--145.
  American Mathematical Society, Providence, RI, 2000.

\bibitem[Gro87]{Gromov}
M.~Gromov.
\newblock Hyperbolic groups.
\newblock In {\em Essays in group theory}, volume~8 of {\em Math. Sci. Res.
  Inst. Publ.}, pages 75--263. Springer, New York, 1987.

\bibitem[GT17]{GT}
Sourav {Ghosh} and Nicolaus {Treib}.
\newblock {Affine {A}nosov representations and Proper actions}.
\newblock {\em ArXiv e-prints}, page arXiv:1711.09712, November 2017.

\bibitem[GW12]{GW2}
O.~Guichard and A.~Wienhard.
\newblock Anosov representations: domains of discontinuity and applications.
\newblock {\em Invent. Math.}, 190(2):357--438, 2012.

\bibitem[Hit92]{Hit}
N.~J. Hitchin.
\newblock Lie groups and {T}eichm\"{u}ller space.
\newblock {\em Topology}, 31(3):449--473, 1992.

\bibitem[HPS06]{HPS}
M.W. Hirsch, C.C. Pugh, and M.~Shub.
\newblock {\em Invariant Manifolds}.
\newblock Lecture Notes in Mathematics. Springer Berlin Heidelberg, 2006.

\bibitem[KLP14]{KLP}
Michael {Kapovich}, Bernhard {Leeb}, and Joan {Porti}.
\newblock {Morse actions of discrete groups on symmetric space}.
\newblock {\em arXiv e-prints}, page arXiv:1403.7671, Mar 2014.

\bibitem[Kna02]{Knapp}
Anthony~W. Knapp.
\newblock {\em Lie groups beyond an introduction}, volume 140 of {\em Progress
  in Mathematics}.
\newblock Birkh\"{a}user Boston, Inc., Boston, MA, second edition, 2002.

\bibitem[KP13]{KraPa}
Steven~G. Krantz and Harold~R. Parks.
\newblock {\em The implicit function theorem}.
\newblock Modern Birkh\"{a}user Classics. Birkh\"{a}user/Springer, New York,
  2013.
\newblock History, theory, and applications, Reprint of the 2003 edition.

\bibitem[KP22]{KP}
Fanny Kassel and Rafael Potrie.
\newblock Eigenvalue gaps for hyperbolic groups and semigroups.
\newblock {\em J. Mod. Dyn.}, 18:161--208, 2022.

\bibitem[KSon]{KaSm}
Fanny Kassel and Ilia Smilga.
\newblock Affine properness criterion for anosov representations and
  generalizations.
\newblock in preparation.

\bibitem[Lab01]{Labourie2}
F.~Labourie.
\newblock Fuchsian affine actions of surface groups.
\newblock {\em J. Differential Geom.}, 59(1):15--31, 09 2001.

\bibitem[Lab06]{Labourie}
F.~Labourie.
\newblock Anosov flows, surface groups and curves in projective space.
\newblock {\em Invent. Math.}, 165(1):51--114, 2006.

\bibitem[Lab22]{Labourie3}
Fran\c{c}ois Labourie.
\newblock Entropy and affine actions for surface groups.
\newblock {\em J. Topol.}, 15(3):1017--1033, 2022.

\bibitem[{Liv}72]{Liv}
A.~N. {Livshits}.
\newblock {Kohomologien dynamischer Systeme.}
\newblock {\em {Izv. Akad. Nauk SSSR, Ser. Mat.}}, 36:1296--1320, 1972.

\bibitem[LW18]{LW}
Fran\c{c}ois Labourie and Richard Wentworth.
\newblock Variations along the {F}uchsian locus.
\newblock {\em Ann. Sci. \'{E}c. Norm. Sup\'{e}r. (4)}, 51(2):487--547, 2018.

\bibitem[Mar83]{Margulis1}
G.~A. Margulis.
\newblock Free completely discontinuous groups of affine transformations.
\newblock {\em Dokl. Akad. Nauk SSSR}, 272(4):785--788, 1983.

\bibitem[Mar84]{Margulis2}
G.~A. Margulis.
\newblock Complete affine locally flat manifolds with a free fundamental group.
\newblock {\em Zap. Nauchn. Sem. Leningrad. Otdel. Mat. Inst. Steklov. (LOMI)},
  134:190--205, 1984.
\newblock Automorphic functions and number theory, II.

\bibitem[Mes07]{Mess}
G.~Mess.
\newblock Lorentz spacetimes of constant curvature.
\newblock {\em Geom. Dedicata}, 126:3--45, 2007.

\bibitem[Mil77]{Milnor}
J.~Milnor.
\newblock On fundamental groups of complete affinely flat manifolds.
\newblock {\em Advances in Math.}, 25(2):178--187, 1977.

\bibitem[Min05]{Mineyev}
I.~Mineyev.
\newblock Flows and joins of metric spaces.
\newblock {\em Geometry and Topology}, 9:403--482, 2005.

\bibitem[PS17]{PotSam}
Rafael Potrie and Andr\'{e}s Sambarino.
\newblock Eigenvalues and entropy of a {H}itchin representation.
\newblock {\em Invent. Math.}, 209(3):885--925, 2017.

\bibitem[{Sam}24]{Samba2}
Andr{\'e}s {Sambarino}.
\newblock {Asymptotic properties of infinitesimal characters and applications}.
\newblock {\em arXiv e-prints}, page arXiv:2406.06250, June 2024.

\bibitem[Smi16]{Smilga}
I.~Smilga.
\newblock Proper affine actions on semisimple {L}ie algebras.
\newblock {\em Ann. Inst. Fourier (Grenoble)}, 66(2):785--831, 2016.

\bibitem[Smi18]{Smilga3}
Ilia Smilga.
\newblock Proper affine actions in non-swinging representations.
\newblock {\em Groups Geom. Dyn.}, 12(2):449--528, 2018.

\bibitem[Smi22]{Smilga4}
Ilia Smilga.
\newblock Proper affine actions: a sufficient criterion.
\newblock {\em Math. Ann.}, 382(1-2):513--605, 2022.

\bibitem[Tit71]{Tit}
J.~Tits.
\newblock Repr\'{e}sentations lin\'{e}aires irr\'{e}ductibles d'un groupe
  r\'{e}ductif sur un corps quelconque.
\newblock {\em J. Reine Angew. Math.}, 247:196--220, 1971.

\bibitem[Tom16]{Tom}
George Tomanov.
\newblock Properly discontinuous group actions on affine homogeneous spaces.
\newblock {\em Tr. Mat. Inst. Steklova}, 292(Algebra, Geometriya i Teoriya
  Chisel):268--279, 2016.

\bibitem[Whi72]{Whit}
Hassler Whitney.
\newblock {\em Complex analytic varieties}.
\newblock Addison-Wesley Publishing Co., Reading, Mass.-London-Don Mills, Ont.,
  1972.

\end{thebibliography}
	\bibliographystyle{alpha}
	
\end{document}